\documentclass[a4paper,10pt]{article}

\usepackage{amssymb,amsfonts}
\usepackage{amsmath,amsfonts,amsthm}
\usepackage[dvipdfmx]{graphicx}
\usepackage[dvipdfmx]{hyperref}
\usepackage{cite}
\usepackage{caption}
\usepackage{mathtools}
\usepackage{booktabs}
\usepackage[T1]{fontenc}
\usepackage[utf8]{inputenc}
\newcommand{\pd}[2]{\frac{\partial #1}{\partial #2}}
\newcommand{\Dt}[1]{\frac{\textnormal{D} #1}{\textnormal{D}t}}
\newcommand{\di}{\textnormal{div\,}}
\newcommand{\D}[1]{\textnormal{D}(#1)}
\newcommand{\vcko}{\mathbf{v}}
\newcommand{\ucko}{\mathbf{u}}

\newcommand{\wcko}{\mathbf{w}}
\newcommand{\nko}{\mathbf{n}}
\newcommand{\Ccko}{\mathbf{C}}
\newcommand{\Dcko}{\mathbf{D}}
\newcommand{\Icko}{\mathbf{I}}
\newcommand{\Fko}{\mathbf{F}}
\newcommand{\fko}{\mathbf{f}}
\newcommand{\gcko}{\mathbf{g}}
\newcommand{\tr}[1]{\textnormal{tr}\, #1}
\newcommand{\trC}{\textnormal{tr}\, \mathbf{C}}

\newcommand{\n}[2]{\left\lVert{#2}\right\rVert_{#1}}
\newcommand{\trian}{\mathcal{T}_h}
\newcommand{\au}{\ucko_h^n}
\newcommand{\aou}{\ucko_h^{n-1}}
\newcommand{\ap}{p_h^n}

\newcommand{\aC}{\Ccko_h^n}
\newcommand{\aoC}{\Ccko_h^{n-1}}

\newcommand{\af}{\fko_h^n}
\newcommand{\aF}{\Fko_h^n}

\newcommand{\SPu}{\hat{\ucko}_h}
\newcommand{\SPp}{\hat{p}_h}

\renewcommand{\(}{\left(}
\renewcommand{\)}{\right)}
\newcommand{\defeq}{\vcentcolon=}

\newtheorem{defin}{Definition}
\newtheorem{theo}{Theorem}
\newtheorem{prop}{Proposition}
\newtheorem{hypo}{Hypothesis}
\newtheorem{lemma}{Lemma}
\newtheorem{remark}{Remark}
\newtheorem*{example}{Example}
\newcommand{\fz}{\frac}
\newcommand{\prz}[2]{ \frac{\partial{#1}}{\partial{#2}} }
\newcommand{\pz}{\partial}
\newcommand{\lA}{\langle}
\newcommand{\rA}{\rangle}

\newcommand{\ol}{\overline}
\newcommand{\barO}{\bar{\Omega}}
\newcommand{\disp}{\displaystyle}
\renewcommand{\Omega}{\varOmega}
\renewcommand{\Gamma}{\varGamma}
\renewcommand{\Psi}{\varPsi}
\renewcommand{\Pi}{\varPi}
\newcommand{\ecko}{\mathbf{e}}
\newcommand{\rcko}{\mathbf{r}}
\newcommand{\Ecko}{\mathbf{E}}
\newcommand{\Rcko}{\mathbf{R}}
\newcommand{\etacko}{{\boldsymbol\eta}}
\newcommand{\Xicko}{{\boldsymbol {\rm\Xi}}}
\allowdisplaybreaks[4]
\DeclareCaptionLabelFormat{tableandpicture}{ \tablename\, \& #1~#2}

\providecommand{\keywords}[1]{\textit{Keywords:} #1}
\providecommand{\msc}[1]{\textit{2010 MSC:} #1}
\title{
Numerical analysis of the {Oseen}-type {Peterlin} viscoelastic model by the stabilized {Lagrange}--{Galerkin} method\\
Part~I: A nonlinear scheme
}
\author{
{M\'{a}ria Luk\'{a}\v{c}ov\'{a}-Medvid'ov\'{a}}$^{1}$,
{Hana~Mizerov\'{a}}$^{1}$,\\
{Hirofumi~Notsu}$^{2,3}$
and
{Masahisa~Tabata}$^{4}$
\bigskip\\
\normalsize
$^1$ Institute of Mathematics, University of Mainz, Mainz 55099, Germany \\
\normalsize
$^2$ Faculty of Mathematics and Physics, Kanazawa University, Kanazawa 920-1192, Japan \\
\normalsize
$^3$ Japan Science and Technology Agency~(JST), PRESTO, Saitama 332-0012, Japan\\
\normalsize
$^4$ Department of Mathematics, Waseda University, Tokyo 169-8555, Japan
}
\date{}
\begin{document}
\maketitle
\begin{abstract}
We present a nonlinear stabilized {Lagrange}--{Galerkin} scheme for the {Oseen}-type {Peterlin} viscoelastic model.
Our scheme is a combination of the method of characteristics and {Brezzi}--{Pitk\"{a}ranta}'s stabilization method for the conforming linear elements,
which yields an efficient computation with a small number of degrees of freedom.
We prove error estimates with the optimal convergence order without any relation between the time increment and the mesh size.
The result is valid for both the diffusive and non-diffusive models for the conformation tensor in two space dimensions.
We introduce an additional term that yields a suitable structural property and allows us to obtain required energy estimate. 
The theoretical convergence orders are confirmed by numerical experiments.
\par
In a forthcoming paper, Part~II, a linear scheme is proposed and the corresponding error estimates are proved in two and three space dimensions for the diffusive model.
\smallskip\\
\msc{65M12, 76A05, 65M60, 65M25}
\end{abstract}
%
%
%
%
%
%
\section{Introduction}\label{sec:intro}
%
In the daily life we encounter many biological, industrial or geological fluids that do not satisfy the {Newtonian} assumption, i.e., the linear dependence between the stress tensor and the deformation tensor.
These fluids belong to the class of the {non-Newtonian} fluids.
In order to describe such complex fluids the stress tensor is represented as a sum of the viscous~({Newtonian}) part and the extra stress due to the polymer contribution.
\par
In literature we can find several models that are employed to describe various aspects of complex viscoelastic fluids.
One of the well-known viscoelastic models is the {Oldroyd-B} model, which is derived from the {Hookean} dumbbell model with a linear spring force law.
The model is  a system of equations for the velocity, the pressure and the extra stress tensor, cf., e.g.,~\cite{Ren-2000,Ren-2008}.
\par
Numerical schemes for  the {Oldroyd-B} type models have been studied by many authors.
For example, we can find a finite difference scheme based on the reformulation of the equation for the extra stress tensor by using the log-conformation representation in {Fattal} and {Kupferman}~\cite{FatKup-2004,FatKup-2005},
free energy dissipative {Lagrange}--{Galerkin} schemes with or without the log-conformation representation in {Boyaval} et al.~\cite{BoyLelMan-2009},
finite element schemes using the idea of the generalized {Lie} derivative in {Lee} and {Xu}~\cite{LeeXu-2006} and {Lee} et al.~\cite{LeeXuZha-2011},
and further related numerical schemes and computations in~\cite{AboMatWeb-2002,BonPicLas-2006,CroKeu-1982,Keu-1986,MarCro-1987,NadSeq-2007,LNS-2015,WapKeuLeg-2000} and references therein.
To the best of our knowledge, however, there are no results on error estimates of numerical schemes for the {Oldroyd-B} model.
As for the simplified {Oldroyd-B} model with no convection terms
{Picasso} and {Rappaz}~\cite{PicRap-2001} and {Bonito} et al.~\cite{BonClePic-2007} have given error estimates for stationary and non-stationary problems, respectively.
The development of stable and convergent numerical methods for  the {Oldroyd-B} type models, especially in the elasticity-dominated case, is still an active research area.
\par
In this paper, Part~I, and the forthcoming paper~\cite{LMNT-Peterlin_Oseen_Part_I}, Part II, we consider the so-called {Peterlin} viscoelastic model, which is a system of the flow equations and an equation for the conformation tensor, cf.~\cite{Ren-2000,Ren-2008}.
In~\cite{Pet-1966} Peterlin proposed a mean-field closure according to which the average of the elastic force over thermal fluctuations is replaced by the value of the force at the mean-squared polymer extension.
More precisely, instead of  the nonlinear spring force law~$F(R) = \gamma(|R|^2) R$ that acts in polymer dumbbells the Peterlin approximation ${F}({R})\approx \gamma( \langle |{R}|^2 \rangle ){R}$ is applied, where $R$ is the vector connecting the dumbbell beads and $\gamma$ is the spring constant.
That means, that the length of the spring in the spring constant $\gamma$ is replaced by the average length of the  spring $\langle |{R}|^2 \rangle \equiv \tr \Ccko$.
Consequently, we can derive an evolution equation for the conformation tensor $\Ccko$, which is in a closed form, cf.~\cite{Ren-2008, Ren-2000, RenWan-2015, LukMizNec-2015, Miz-2015}.
Note that in literature one can also find the Peterlin approximation in the context of finitely extensible nonlinear elastic (FENE) dumbbell model, which was subsequently termed the FENE-P model, cf.~\cite{BirDotJoh-1980}. In this model the denominator of the FENE force of the corresponding kinetic model is replaced  by the mean value of the elongation yielding the macroscopic FENE-P model.
On the other hand, {Renardy} recently proposed a general macroscopic constitutive model, that is motivated by Peterlin dumbbell theories with a nonlinear spring law for an infinitely extensible spring, see {Renardy}~\cite{RenWan-2015, Ren-2010} and a recent paper by {Luk\'{a}\v{c}ov\'{a}-Medvi\v{d}ov\'{a}} et al.~\cite{LMNR-2016}, where the global existence of weak solutions has been obtained.
The diffusive {Peterlin} viscoelatisc model studied in the present paper has been obtained by a particular choice of these general constitutive functions. 
This model has been studied analytically by {Luk\'{a}\v{c}ov\'{a}-Medvi\v{d}ov\'{a}} et al.~\cite{LukMizNec-2015}, where the global existence of weak solutions and the uniqueness of regular solutions have been proved.
Let us mention that, even when the velocity field is given, the equation for the conformation tensor in the {Peterlin} model is still nonlinear, while the {Oldroyd-B} model is linear with respect to the extra stress tensor.
Hence, we can say that the nonlinearity of the {Peterlin} model is stronger than that of the {Oldroyd-B} model.
As a starting point of the numerical analysis of the {Peterlin} model, we consider the {Oseen}-type model, where the velocity of the material derivative is replaced by a known one, in order to concentrate on the treatment of nonlinear terms arising from the elastic stress.
\par
Our aim is to develop a stabilized {Lagrange}--{Galerkin} method for the {Peterlin} viscoelastic model.
It consists of the method of characteristics and {Brezzi}--{Pitk\"{a}ranta}'s stabilization method~\cite{BrePit-1984} for the conforming linear elements.
The method of characteristics yields the robustness in convection-dominated flow problems,
and the stabilization method reduces the number of degrees of freedom in computation.
In our recent works by {Notsu} and {Tabata}~\cite{NT-2016-M2AN,NT-2015-JSC,NT-NCP} the stabilized {Lagrange}--{Galerkin} method has been applied successfully for the {Oseen}, {Navier}--{Stokes} and natural convection problems and optimal error estimates have been proved.
\par
We establish the numerical analysis of the stabilized {Lagrange}--{Galerkin} method for the {Oseen}-type {Peterlin} model in this paper, Part~I, and the forthcoming paper~\cite{LMNT-Peterlin_Oseen_Part_I}, Part~II.
The results of the two papers are summarized in Tables~\ref{table:summary} and~\ref{table:summary_dt}, where $\varepsilon$ is the diffusion coefficient in the equation for the conformation tensor, $d$ is the spatial dimension, $h$ is the representative mesh size and $\Delta t$ is the time increment.
\par
In Part~I, a nonlinear stabilized {Lagrange}--{Galerkin} scheme for the diffusive~($\varepsilon > 0$) and the non-diffusive~($\varepsilon = 0$) {Peterlin} model is presented and error estimates with the optimal convergence order are proved without any relation between discretization parameters~$\Delta t$ and~$h$ in two dimensions.
For the proof we rely on a key lemma, cf. Lemma~\ref{lem:vanish_nonlin}, in which a special structural property using  an additional term $( \di \au (\aC)^{\#},\Dcko_h )$ is shown.
However, this property does not hold in three-dimensional case.
This is the reason why the convergence result is shown only in two space dimensions.
The theoretical convergence orders are confirmed by numerical experiments.
Since the scheme is nonlinear, the existence and uniqueness of the scheme are studied additionally, and we show that the scheme has a solution without any relation between $h$ and $\Delta t$ and that the solution is unique for the diffusive and the non-diffusive cases under the conditions $\Delta t = O( 1/(1+ |\log h|)^2 )$ and $\Delta t = O( h )$, respectively, in two dimensions.
\par
In Part II a linear scheme for the diffusive model is presented and optimal error estimates are proved under mild stability conditions, $\Delta t = O( 1/ \sqrt{1+ |\log h|} \, )$ and $\Delta t = O( \sqrt{h} \, )$, in two and three dimensions, respectively.
Moreover, the existence and uniqueness of its numerical solution are shown as well.
The theoretical convergence orders are again confirmed  by numerical experiments.
\begin{table}[!h]
\centering
\caption{
Summary of our results in Part~I and Part~II.\
($\varepsilon$ is the diffusion coefficient for the conformation tensor and $d$ is the spatial dimension.)
}
\label{table:summary}
\begin{tabular}{crcrc}
\toprule
\phantom{{\LARGE $|$}}
&& Part~I && Part~II \\
\hline
\phantom{{\Huge $|$}}
Scheme
\phantom{{\Huge $|$}}
&& Nonlinear && Linear \\
\phantom{{\Huge $|$}}
$\varepsilon$
\phantom{{\Huge $|$}}
&& $ \ge 0$ && $ > 0$ \\
\phantom{{\Huge $|$}}
$d$
\phantom{{\Huge $|$}}
&& $2$ && $2$ and $3$ \\
\bottomrule
\end{tabular}
\end{table}
\begin{table}[!h]
\centering
\caption{
Conditions on the time increment~$\Delta t$ with respect to the mesh size~$h$.\
($\varnothing$  means that no condition is required.)
}
\label{table:summary_dt}
\begin{tabular}{crcrc}
\toprule
\phantom{{\Huge $I$}}
&& Part~I,\;\; $d=2$ && Part~II,\;\; $\varepsilon > 0$
\\
\cline{3-3}
\cline{5-5}
\phantom{{\Huge $|$}}
Existence
\phantom{{\Huge $|$}}
&&
$\varnothing$
&&
$\varnothing$
\\
Uniqueness
&&
\begin{minipage}[b]{5cm}
\centering
\vspace{1em}
\begin{tabular}{cc}
$\varepsilon > 0$
&
$\varepsilon = 0$
\\
\hline
$\disp O\Bigl( \fz{1}{(1+ |\log h|)^2} \Bigr)$
&
\phantom{$\Biggl|$}
$O(h)$
\phantom{$\Biggl|$}
\end{tabular}
\end{minipage}
&&
\begin{minipage}[c]{1cm}
\centering
$\varnothing$
\end{minipage}
\\
\begin{minipage}[c]{1.5cm}
\centering
Optimal error estimates
\end{minipage}
&&
$\varnothing$
&&
\begin{minipage}[c]{5cm}
\centering
\begin{tabular}{cc}
\phantom{{\Huge$I$}}
$d = 2$
\phantom{{\Huge$I$}}
&
$d = 3$
\\
\hline
$\disp O\Bigl( \fz{1}{\sqrt{1 + |\log h|}}\Bigr)$
&
\phantom{$\Biggl|$}
$O\bigl(\sqrt{h}\;\bigr)$
\phantom{$\Biggl|$}
\end{tabular}
\end{minipage}
\medskip
\\
\bottomrule
\end{tabular}
\end{table}
\par
Let us summarize that in both papers, Part~I~(nonlinear scheme) and Part~II~(linear scheme), we present the results for optimal error estimates (i)~for the non-diffusive case~$(\varepsilon = 0)$ in two space dimensions and (ii)~for the diffusive case~$(\varepsilon > 0)$ in three space dimensions, respectively.
\par
As mentioned in {Boyaval} et al.~\cite{BoyLelMan-2009}, the positive definiteness of the conformation tensor is important in the analysis of numerical schemes for the Oldroyd-B model and has been overcome by using, e.g., the log-conformation representation in {Fattal} and {Kupferman}~\cite{FatKup-2004,FatKup-2005}.
While some schemes preserving the positive definiteness have been developed, there are, as far as we know, no convergence results of such schemes.
In our papers, Part~I and Part~II, we have obtained the convergence results without any assumption on the positive definiteness. This is an additional feature of our proof.
\vspace{0.5em}
\par
The paper is organized as follows.
In Section~\ref{sec:model} the mathematical formulation of the Oseen-type {Peterlin} viscoelastic model is described.
In Section~\ref{sec:scheme} a nonlinear stabilized {Lagrange}--{Galerkin} scheme is presented.
The main result on the convergence with optimal error estimates is stated in Section~\ref{sec:main_result}, and proved in Section~\ref{sec:proofs}.
In Section~\ref{sec:uniqueness} uniqueness of the numerical solution is shown.
Theoretical order of convergence is confirmed by numerical experiments in Section~\ref{sec:numerics}.
%
%
%
%
%
%
%
%
%
%
%
%
%
%
%
%
%
%
%
\section{The {Oseen}-type {Peterlin} viscoelastic model}\label{sec:model}
The function spaces and the notation to be used throughout the paper are as follows.
Let $\Omega$ be a bounded domain in $\mathbb{R}^2$, $\Gamma\defeq\pz\Omega$ the boundary of $\Omega$, and~$T$ a positive constant.
For $m \in \mathbb{N}\cup \{0\}$ and $p\in [1,\infty]$ we use the {\rm Sobolev} spaces $W^{m,p}(\Omega)$, $W^{1,\infty}_0(\Omega)$, $H^m(\Omega) \, (=W^{m,2}(\Omega))$, $H^1_0(\Omega)$ and $L^2_0(\Omega)\defeq\{q\in L^2(\Omega); \int_\Omega q\,dx =0\}$.
Furthermore, we employ function spaces $H^m_{sym}(\Omega) \defeq \{\Dcko\in H^m (\Omega)^{2\times 2};~\Dcko = \Dcko^T\}$ and $C^m_{sym}(\barO) \defeq C^m(\barO)^{2\times 2}\cap H^m_{sym}(\Omega)$, where the superscript~$T$ stands for the transposition.
For any normed space $S$ with norm $\|\cdot\|_S$, we define function spaces $H^m(0,T; S)$ and $C([0,T]; S)$ consisting of $S$-valued functions in $H^m(0,T)$ and $C([0,T])$, respectively.
We use the same notation $(\cdot, \cdot)$ to represent the $L^2(\Omega)$ inner product for scalar-, vector- and matrix-valued functions.
The dual pairing between $S$ and the dual space $S^\prime$ is denoted by $\lA\cdot, \cdot\rA$.
The norms on $W^{m,p}(\Omega)$ and $H^m(\Omega)$ and their seminorms are simply denoted by $\|\cdot\|_{m,p}$ and $\|\cdot\|_m \, (= \|\cdot\|_{m,2})$ and by $|\cdot|_{m,p}$ and $|\cdot|_m \, (= |\cdot|_{m,2})$, respectively.
The notations~$\|\cdot\|_{m,p}$, $|\cdot|_{m,p}$, $\|\cdot\|_m$ and $|\cdot|_m$ are employed not only for scalar-valued functions but also for vector- and matrix-valued ones.
We also denote the norm on $H^{-1}(\Omega)^2$ by $\|\cdot\|_{-1}$.
For $t_0$ and $t_1\in\mathbb{R}$ we introduce the function space,
\begin{align*}
Z^m(t_0, t_1) \defeq \bigl\{ \psi \in H^j(t_0, t_1; H^{m-j}(\Omega));~j=0,\ldots,m,\ \|\psi\|_{Z^m(t_0, t_1)} < \infty \bigr\}
\end{align*}
with the norm
\begin{align*}
\|\psi\|_{Z^m(t_0, t_1)} \defeq \biggl\{ \sum_{j=0}^m \|\psi\|_{H^j(t_0,t_1; H^{m-j}(\Omega))}^2 \biggr\}^{1/2},
\end{align*}
and set $Z^m \defeq Z^m(0, T)$.
We often omit $[0,T]$, $\Omega$, and the superscripts~$2$ and~$2\times 2$ for the vector and the matrix if there is no confusion, e.g., we shall write $C(L^\infty)$ in place of $C([0,T]; L^\infty(\Omega)^{2\times 2})$.
For square matrices $\mathbf{A}$ and $\mathbf{B} \in \mathbb{R}^{2\times 2}$ we use the notation $\mathbf{A}:\mathbf{B} = \sum_{i,j} A_{ij} B_{ij}$.
\par
We consider the system of equations describing the unsteady motion of an incompressible viscoelastic fluid,
\begin{subequations}\label{model}
\begin{align}
\Dt{\ucko} - \di\bigl( 2\nu\D{\ucko} \bigr) + \nabla p & =\di [(\trC) \Ccko] + \fko & & \mbox{in}~\Omega \times (0,T),  \label{model_ucko}\\
\di \ucko &= 0 & & \mbox{in}~\Omega \times (0,T), \\
\Dt{\Ccko} - \varepsilon\Delta\Ccko =  (\nabla\ucko)\Ccko & + \Ccko(\nabla\ucko)^T - \(\trC\)^2 \Ccko   + (\trC) \Icko + \Fko & & \mbox{in}~\Omega \times (0,T),
\label{model_Ccko}\\
\ucko &= \mathbf{0}, \quad \varepsilon \pd{\Ccko}{\nko} = \mathbf{0}, & & \mbox{on}~\Gamma \times (0,T),
\label{model_bc}\\
\ucko &= \ucko^0,\quad \Ccko =   \Ccko^0,
 & & \mbox{in}~\Omega,\ \mbox{at}\ t=0,
\label{model_ic}
\end{align}
\end{subequations}
where $\(\ucko, p, \Ccko\): \Omega\times (0,T) \rightarrow \mathbb{R}^2 \times \mathbb{R}\times \mathbb{R}^{2\times 2}_{sym}$ are the unknown velocity, pressure and conformation tensor, $\nu >0$ is a fluid viscosity, $\varepsilon \in [0, 1]$ is an elastic stress viscosity,
$(\fko, \Fko): \Omega \times (0,T) \rightarrow \mathbb{R}^2 \times \mathbb{R}^{2\times 2}$ is a pair of given external forces,
$\nabla \ucko$ is the (matrix-valued) velocity gradient defined by $(\nabla \ucko)_{ij} \defeq \pz u_i/\pz x_j~(i, j=1,2)$, $\D{\ucko} \defeq (1/2) [\nabla \ucko + (\nabla \ucko)^T]$ is the symmetric part of the velocity gradient,
$\Icko$~is the identity matrix,
$\nko:\Gamma\to\mathbb{R}^2$~is the outward unit normal,
$(\ucko^0, \Ccko^0 ) : \Omega\to \mathbb{R}^2 \times \mathbb{R}^{2 \times 2}_{sym}$ is a pair of given initial functions,
and $\textnormal{D}/\textnormal{D}t$ is the material derivative defined by
\begin{align*}
\Dt{\ } \defeq \pd{\ }{t} + \wcko\cdot\nabla,
\end{align*}
where $\wcko:\Omega\times (0,T) \rightarrow \mathbb{R}^2$ is a given velocity.
\begin{remark}
(i)~In this paper we pay attention to the dependency on~$\varepsilon$ to include the degenerate case~$\varepsilon = 0$.
The upper bound~$1$ of $\varepsilon$ is not essential but replaced by any positive constant~$\varepsilon_0$, i.e., $\varepsilon \in [0,\varepsilon_0]$.
The upper bound is needed in choosing the constants~$h_0$, $\Delta t_0$ and $c_\dagger$ independent of~$\varepsilon$ in Theorem~\ref{thm:error_estimates} below, where it is used for the estimate~\eqref{ieq:R3} in Lemma~\ref{lem:estimates_r_R}.
\smallskip\\
(ii)~When $\varepsilon>0$, under regularity condition on~$\wcko$ the global existence of a weak solution of~\eqref{weak_formulation} below can be proved in a similar way to the fully nonlinear case~\cite{LukMizNec-2015}.
\smallskip\\
(iii)~When $\varepsilon=0$,  there is neither the diffusion term in~\eqref{model_Ccko} nor the boundary condition on~$\Ccko$ in~\eqref{model_bc}.
Because of the loss of the ellipticity, $\Ccko(t)$ does not belong to~$H^1(\Omega)^{2\times 2}$ in general.
If there exists a solution satisfying Hypothesis~\ref{hyp:regularity} below, then we can show the convergence of the finite element solution to the exact one in Theorem~\ref{thm:error_estimates}.
\end{remark}
\par
We formulate an assumption for the given velocity~$\wcko$.
\begin{hypo}\label{hyp:w}
The function $\wcko$ satisfies $\wcko \in C([0,T];W^{1,\infty}_0(\Omega)^2).$
\end{hypo}
\par
Let $V \defeq H_0^1(\Omega)^2$, $Q \defeq L_0^2(\Omega)$ and $W \defeq H^1_{sym}(\Omega)$.
We define the bilinear forms $a_u$ on $V \times V,$  $b$ on $V \times Q,$ $\mathcal{A}$ on $(V \times Q)\times(V \times Q)$ and $a_c$ on $W \times W$ by
\begin{align*}
a_u\(\ucko,\vcko\) & \defeq  2 \bigl( \D{\ucko}, \D{\vcko} \bigr), & b(\ucko,q) & \defeq - (\di \ucko, q), & \mathcal{A}\bigl( (\ucko,p), (\vcko,q) \bigr) & \defeq \nu a_u\( \ucko,\vcko \) + b(\ucko,q) + b(\vcko,p), \\
a_c\(\Ccko,\Dcko\) & \defeq (\nabla\Ccko, \nabla\Dcko),
\end{align*}
respectively.
We present the weak formulation of the problem~\eqref{model};
find $(\ucko,p,\Ccko): (0,T) \rightarrow V \times Q \times W$ such that for $t \in (0,T)$
\begin{subequations}\label{weak_formulation}
\begin{align}
\biggl( \Dt{\ucko}(t),\vcko \biggr) & + \mathcal{A}\bigl( (\ucko,p)(t), (\vcko,q) \bigr) = -  \(\trC(t)\,\Ccko(t),\nabla\vcko\) + \(\fko(t),\vcko\),
\label{weak_formulation_ucko}\\
\biggl( \Dt{\Ccko}(t), \Dcko \biggr) & + \varepsilon a_c \bigl( \Ccko(t),\Dcko \bigr) = 2\bigl( (\nabla\ucko(t)) \Ccko(t), \Dcko \bigr) - \bigl( (\trC(t))^2 \Ccko(t),\Dcko \bigr) +\(\trC(t)\Icko,\Dcko\) + \(\Fko(t),\Dcko\),
\label{weak_formulation_Ccko} \\
&
\qquad\qquad\qquad\qquad\qquad\qquad\qquad\qquad\qquad\qquad\qquad\qquad\qquad\qquad
\forall (\vcko,q,\Dcko) \in V\times Q \times W, \notag
\end{align}
\end{subequations}
with $( \ucko(0),  \Ccko(0) ) =( \ucko^0 , \Ccko^0)$.
%
%
%
%
%
%
%
%
%
\section{A nonlinear stabilized {Lagrange}--{Galerkin} scheme}\label{sec:scheme}
The aim of this section is to present a nonlinear stabilized {Lagrange}--{Galerkin} scheme for~\eqref{model}.
\par
Let $\Delta t$ be a time increment, $N_T \defeq \lfloor T/\Delta t \rfloor$ the total number of time steps and $t^n \defeq n \Delta t$ for $n=0,\ldots,N_T$.
Let $\gcko$ be a function defined in $\Omega\times (0,T)$ and $\gcko^n \defeq \gcko(\cdot,t^n)$.
For the approximation of the material derivative we employ the first-order characteristics method,
\begin{align}\label{approx_matder}
\Dt{\gcko}(x,t^n) = \frac{\gcko^n(x) - \( \gcko^{n-1} \circ X_1^n\) (x)}{\Delta t} + O(\Delta t),
\end{align}
where $X_1^n:\Omega \to \mathbb{R}^2$ is a mapping defined by
\[
X_1^n(x) \defeq x-\wcko^n(x)\Delta t,
\]
and the symbol~$\circ$ means the composition of functions,
\[
(\gcko^{n-1}\circ X_1^n)(x) \defeq \gcko^{n-1} ( X_1^n(x) ).
\]
For the details on deriving the approximation \eqref{approx_matder} of $\textnormal{D} \gcko/ \textnormal{Dt},$ see, e.g.,~\cite{NT-2015-JSC}.
The point $X_1^n(x)$ is called the upwind point of~$x$ with respect to $\wcko^n$.
The next proposition, which is a direct consequence of \cite{RuiTab-2002} and \cite{TabUch-2015-NS}, presents sufficient conditions to ensure that all upwind points defined by $X_1^n$ are in $\Omega$ and that its Jacobian~$J^n \defeq \det ( \pz X_1^n / \pz x )$ is around $1$.
\begin{prop}\label{prop:RT_TU}
Suppose Hypothesis~\ref{hyp:w} holds.
Then, we have the following for $n \in \{0,\ldots,N_T\}$.
\smallskip\\
(i)~Under the condition~$\Delta t |\wcko|_{C(W^{1,\infty})} < 1$, $X_1^n: \Omega \to \Omega$ is bijective.
\smallskip\\
(ii)~Furthermore, under the condition
\begin{align}
\Delta t |\wcko|_{C(W^{1,\infty})} \le 1/4,
\label{cond:dt_w_Jacobian}
\end{align}
the estimate $1/2 \le J^n \le 3/2$ holds.
\end{prop}
\par
For the sake of simplicity we suppose that $\Omega$ is a polygonal domain.
Let $\trian=\{K\}$ be a triangulation of $\bar{\Omega} \ (= \bigcup_{K\in\mathcal{T}_h} K )$, $h_K$ the diameter of $K \in \trian$ and $h \defeq \max_{K\in\trian}h_K$ the maximum element size.
We consider a regular family of subdivisions $\{\trian\}_{h\downarrow 0}$ satisfying the inverse assumption~\cite{Cia-1978}, i.e., there exists a positive constant $\alpha_0$ independent of $h$ such that
\begin{align*}
\frac{h}{h_K}\leq \alpha_0, \quad \forall K \in \trian, \ \forall h.
\end{align*}
We define the discrete function spaces $X_h$, $V_h$, $M_h$, $Q_h$ and $W_h$ by
\begin{align*}
X_h &\defeq  \left\{ \vcko_h \in C(\bar{\Omega})^2 ; \ \vcko_{h|K} \in P_1(K)^2, \forall K \in \trian \right\}, & V_h &\defeq  X_h \cap V, &\\
M_h &\defeq  \left\{ q_h \in C(\bar{\Omega}) ; \ q_{h|K} \in P_1(K), \forall K \in \trian \right\}, & Q_h &\defeq  M_h \cap Q,&\\
W_h &\defeq  \left\{ \Dcko_h \in C_{sym}(\bar{\Omega}); \ \Dcko_{h|K} \in P_1(K)^{2\times 2}, \forall K \in \trian \right\},
\end{align*}
respectively, where $P_1(K)$ is the polynomial space of linear functions on~$K\in\trian$.
\par
Let $\delta_0$ be a small positive constant fixed arbitrarily and $(\cdot, \cdot)_K $ the $L^2(K)^2$ inner product.
We define the bilinear forms $\mathcal{A}_h$ on $(V \times H^1(\Omega)) \times (V \times H^1(\Omega))$ and $\mathcal{S}_h$ on $H^1(\Omega) \times H^1(\Omega)$ by
\begin{align*}
\mathcal{A}_h\( (\ucko,p),(\vcko,q) \) &\defeq \nu a_u\( \ucko,\vcko \) + b(\ucko,q) + b(\vcko,p) - \mathcal{S}_h(p,q), &
\mathcal{S}_h(p,q) & \defeq \delta_0\sum_{K\in\trian}h_K^2(\nabla p,\nabla q)_K.
\end{align*}
For $\Dcko \in \mathbb{R}^{2 \times 2}_{sym}$ let $\Dcko^{\#} \in \mathbb{R}^{2 \times 2}_{sym}$ be the adjugate matrix of~$\Dcko$ defined by
\begin{align*}
\Dcko^{\#} \defeq \left(
\begin{array}{rr}
D_{22} & -D_{12} \\
-D_{12} & D_{11} \\
\end{array} \right).
\end{align*}
\par
Let $(\fko_h, \Fko_h) \defeq (\{ \fko_h^n\}_{n=1}^{N_T}, \{ \Fko_h^n\}_{n=1}^{N_T} ) \subset L^2(\Omega)^2 \times L^2(\Omega)^{2\times 2}$ and  $(\ucko_h^0,\Ccko_h^0) \in V_h \times W_h$ be given.
A nonlinear stabilized {\rm {Lagrange}--{Galerkin}} scheme for~\eqref{model} is to find $(\ucko_h, p_h, \Ccko_h) \defeq \{(\au,\ap,\aC)\}_{n=1}^{N_T}$ $\subset V_h \times Q_h \times W_h$ such that, for $n=1,\ldots ,N_T$,
\begin{subequations}\label{nonlin_scheme}
\begin{align}
\( \frac{\au-\aou\circ X_1^n}{\Delta t}, \vcko_h \) + \mathcal{A}_h \bigl( ( \au,\ap ), (\vcko_h,q_h) \bigr) & = - \bigl( (\tr\Ccko_h^n)\aC,\nabla \vcko_h \bigr) + ( \af,\vcko_h ),
\label{nonlin_scheme_up}\\
\(\frac{\aC-\aoC \circ X_1^n}{\Delta t},\Dcko_h\) + \varepsilon a_{c}\(\aC,\Dcko_h\) & = 2 \bigl( (\nabla\au)\aC, \Dcko_h \bigr) + \bigl( \di \au (\aC)^{\#},\Dcko_h \bigr) - \bigl( (\tr\Ccko_h^n)^2 \aC, \Dcko_h \bigr) \notag\\
& \quad + \bigl( (\tr\Ccko_h^n)\Icko,\Dcko_h \bigr) + ( \aF, \Dcko_h ),
\label{nonlin_scheme_c} \\
& \qquad\qquad\qquad\qquad\qquad\qquad\quad
 \forall (\vcko_h,q_h,\Dcko_h) \in V_h\times Q_h \times W_h. \notag
\end{align}
\end{subequations}
In  Remark~\ref{rmk:vanish_nonlin} below we show that an additional term, the second term on the right-hand side of~\eqref{nonlin_scheme_c}, is added in order to derive a desired energy inequality.
%
%
%
%
%
%
%
\section{The main result}\label{sec:main_result}
In this section we present the main result on error estimates with the optimal convergence order of scheme~\eqref{nonlin_scheme}.
\par
We use $c$ to represent a generic positive constant independent of the discretization parameters $h$ and $\Delta t$.
We also use constants $c_w$ and $c_s$ independent of $h$ and $\Delta t$ but dependent on~$\wcko$ and the solution~$(\ucko, p, \Ccko)$ of~\eqref{weak_formulation}, respectively, and $c_s$ often depends on~$\wcko$ additionally.
$c$, $c_w$ and~$c_s$ may be dependent on~$\nu$ but are independent of~$\varepsilon$.
The symbol ``$\prime$ (prime)'' is sometimes used in order to distinguish two constants, e.g.,~$c_s$ and~$c_s^\prime$, from each other.
We use the following notation for the norms and seminorms, $\n{V}{\cdot} = \n{V_h}{\cdot} \defeq \n{1}{\cdot}$, $\n{Q}{\cdot} = \n{Q_h}{\cdot} \defeq \n{0}{\cdot}$,
\begin{align*}
\n{Z^2(t_0,t_1)}{(\ucko,\Ccko)} & \defeq \Bigl\{ \n{Z^2(t_0,t_1)}{\ucko}^2 + \n{Z^2(t_0,t_1)}{\Ccko}^2 \Bigr\}^{1/2}, &
\n{\ell^{\infty}(X)}{\ucko} & \defeq \max_{n=0,\ldots,N_T} \n{X}{\ucko^n}, \\
\n{\ell^2(X)}{\ucko} & \defeq \biggl\{ \Delta t \sum_{n=1}^{N_T} \n{X}{\ucko^n}^2 \biggr\}^{1/2}, &
|\ucko|_{\ell^2(X)} & \defeq \biggl\{ \Delta t \sum_{n=1}^{N_T} |\ucko^n|_X^2 \biggr\}^{1/2}, \\
|p|_h & \defeq \biggl\{ \sum_{K \in \trian} h_K^2 ( \nabla p, \nabla p )_K \biggr\}^{1/2}, &
|p|_{\ell^2(|.|_h)} & \defeq \biggl\{ \Delta t \sum_{n=1}^{N_T} |p^n|_h^2 \biggr\}^{1/2},
\end{align*}
for $X=L^2(\Omega)$ or $H^1(\Omega)$.
$\ol{D}_{\Delta t}$ is the backward difference operator defined by $\ol{D}_{\Delta t} u^n \defeq (u^n - u^{n-1})/\Delta t$.
\par
The existence of the solution of scheme~\eqref{nonlin_scheme} is guaranteed by the next proposition whose proof is given in the next section.
\begin{prop}[existence]\label{prop:existence}
Suppose Hypothesis~\ref{hyp:w} holds.
Then for any~$h>0$ and $\Delta t \in (0, 1/2)$ satisfying~\eqref{cond:dt_w_Jacobian}, there exists a solution $(\ucko_h,p_h,\Ccko_h)  \subset V_h \times Q_h \times W_h$ of scheme~\eqref{nonlin_scheme}.
\end{prop}
%
\par
We state the main result after preparing a projection and a hypothesis.
\begin{defin}[{Stokes} projection]
For $(\ucko, p) \in V \times Q$ we define the {\rm Stokes} projection $(\SPu, \SPp) \in V_h \times Q_h$ of $(\ucko, p)$ by
\begin{align}
\mathcal{A}_h\((\SPu,\SPp),(\vcko_h, q_h)\) = \mathcal{A}\((\ucko,p),(\vcko_h, q_h)\), \quad \forall (\vcko_h, q_h) \in V_h \times Q_h.
\label{Stokes_projection}
\end{align}
\end{defin}
The {\rm Stokes} projection derives an operator $\Pi_h^{\rm S}: V \times Q \to V_h\times Q_h$ defined by $\Pi_h^{\rm S} (\ucko, p) := (\hat{\ucko}_h, \hat{p}_h)$.
The first component~$\hat{\ucko}_h$ of $\Pi_h^{\rm S} (\ucko, p)$ is denoted by $[\Pi_h^{\rm S} (\ucko, p)]_1$.
Let $\Pi_h: L^2(\Omega) \to M_h$ be the {Cl\'{e}ment} interpolation operator~\cite{Cle-1975}.
The Cl\'{e}ment operators on $L^2(\Omega)^2$ and $L^2(\Omega)^{2\times 2}$ are denoted by the same symbol~$\Pi_h$.
%
\begin{remark}\label{rmk:no_poisson}
The {\rm Cl\'{e}ment} operator is defined for functions from $L^2(\Omega)$.
When a function belongs to~$C(\bar{\Omega})$, we can replace the {\rm Cl\'{e}ment} operator by the {\rm Lagrange} operator~$\Pi_h^L: C(\bar{\Omega}) \to M_h$.
\end{remark}
\begin{hypo}\label{hyp:regularity}
The solution $(\ucko, p, \Ccko) $ of~\eqref{weak_formulation} satisfies
$\ucko \in
Z^2(0,T)^2 \cap H^1(0,T;V \cap H^2(\Omega)^2)
\cap C([0,T]; W^{1,\infty}(\Omega)^2)$,
$p \in H^1(0,T; Q \cap H^1(\Omega))$ and
\[
\Ccko \in
\left\{
\begin{aligned}
& Z^2(0,T)^{2\times 2} \cap L^2(0,T;W) \cap C([0,T]; H^2(\Omega)^{2\times 2}) && (\varepsilon > 0),\\
& Z^2(0,T)^{2\times 2} \cap L^2(0,T;W) \cap C([0,T]; L^\infty(\Omega)^{2\times 2}) && (\varepsilon = 0).
\end{aligned}
\right.
\]
\end{hypo}
We now impose the conditions
\begin{align}
(\ucko_h^0,\Ccko_h^0) = ([\Pi_h^{\rm S} (\ucko^0, 0)]_1, \Pi_h\Ccko^0), \quad (\fko_h, \Fko_h)=(\fko,\Fko).
\label{cond:if}
\end{align}
%
\begin{theo}[error estimates]\label{thm:error_estimates}
Suppose Hypotheses~\ref{hyp:w} and~\ref{hyp:regularity} hold.
Then, there exist positive constants~$h_0$, $\Delta t_0$ and $c_\dagger$ independent of~$\varepsilon$ such that, for any pair $(h, \Delta t)$ satisfying
\begin{align}
 h\in (0,h_0], \quad \Delta t\in (0,\Delta t_0],
\label{cond:ht}
\end{align}
and any solution~$(\ucko_h, p_h, \Ccko_h)$ of scheme~\eqref{nonlin_scheme} with \eqref{cond:if}, it holds that
\begin{align}
& \|\ucko_h-\ucko\|_{\ell^{\infty}(L^2)},  \ \sqrt{\nu} \|\ucko_h-\ucko\|_{\ell^2(H^1)}, \  |p_h-p|_{\ell^2(|.|_h)}, \qquad\qquad\qquad\qquad \notag\\
& \|\Ccko_h-\Ccko\|_{\ell^{\infty}(L^2)}, \ \sqrt{\varepsilon} |\Ccko_h-\Ccko|_{\ell^2(H^1)}, \ \bigl\| \tr(\Ccko_h-\Ccko) (\Ccko_h-\Ccko) \bigr\|_{\ell^2(L^2)} \le c_{\dagger} (h +\Delta t).
\label{ieq:error_estimates}
\end{align}
\end{theo}
%
\begin{remark}\label{rmk:error_estimates_eps_0}
(i)~The estimates~\eqref{ieq:error_estimates} hold even for $\varepsilon=0$.
Then, of course, the fifth term of the left-hand side of \eqref{ieq:error_estimates} vanishes.
\smallskip\\
(ii)~Here we do not need uniqueness of the solution of scheme~\eqref{nonlin_scheme}.
Uniqueness of the numerical solution will be  discussed later in Proposition~\ref{prop:uniqueness}.
\smallskip\\
(iii)~The positive definiteness of the exact and numerical solutions is not required for the above error estimates.
\end{remark}
%
%
%
%
%
\section{Proofs}\label{sec:proofs}
%
In what follows we prove Proposition~\ref{prop:existence} and Theorem~\ref{thm:error_estimates}.
%
%
%
%
%
\subsection{Preliminaries}
Let us list lemmas directly employed below in the proofs.
In the lemmas, $\alpha_i$, $i=1,\ldots,4$, are numerical constants.
They are independent of~$h$, $\Delta t$, $\nu$ and~$\varepsilon$ but may depend on~$\Omega$.
\begin{lemma}[ \cite{Nec-1967} ]\label{lem:Korn}
Let $\Omega$ be a bounded domain with a {\rm Lipschitz}-continuous boundary.
Then, the following inequalities hold.
\begin{align*}
\|\D{\vcko}\|_0 \le \|\vcko\|_1 \le \alpha_1 \|\D{\vcko}\|_0,\qquad \forall \vcko \in H^1_0(\Omega)^2.
\end{align*}
\end{lemma}
We introduce the function
\begin{align}
D(h) \defeq (1+|\log h|)^{1/2},
\label{Dofh}
\end{align}
which is used in the sequel.
\begin{lemma}[ \cite{BreSco-2008,Cia-1978,Cle-1975} ]\label{lem:interpolation}
The following inequalities hold.
\begin{align}
 \n{0,\infty}{\Pi_h\gcko} &\le \n{0,\infty}{\gcko}, & \forall \gcko &\in L^\infty(\Omega)^s,
\label{ieq:Pi_h_0inf} \\
 \n{1,\infty}{\Pi_h\gcko} &\le \alpha_{20} \n{1,\infty}{\gcko}, & \forall \gcko &\in W^{1,\infty}(\Omega)^s, \notag \\
 \n{0}{\Pi_h \gcko - \gcko} &\le \alpha_{21} h \n{1}{\gcko}, & \forall \gcko & \in H^1(\Omega)^s \cap L^\infty(\Omega)^s, \notag \\
 \n{1}{\Pi_h \gcko - \gcko} &\le \alpha_{22} h \n{2}{\gcko}, & \forall \gcko & \in H^2(\Omega)^s, \notag \\
 \n{0,\infty}{\gcko_h} & \le \alpha_{23} h^{-1} \n{0}{\gcko_h}, & \forall \gcko_h & \in S_h, \notag \\
 \n{0,\infty}{\gcko_h}  & \le \alpha_{24} D(h) \n{1}{\gcko_h}, & \forall \gcko_h & \in S_h, \notag \\
 \n{1,\infty}{\gcko_h}  & \le \alpha_{25} h^{-1} \n{1}{\gcko_h}, & \forall \gcko_h & \in S_h, \notag \\
 \n{1}{\gcko_h}  & \le \alpha_{26} h^{-1} \n{0}{\gcko_h}, & \forall \gcko_h & \in S_h, \notag
\end{align}
where $s=2$ or $2\times 2$ and $S_h=V_h$ or $W_h$.
\end{lemma}
\begin{lemma}[ \cite{BreDou-1988} ]\label{lem:estimates_Stokes_projection}
Assume $(\ucko, p) \in (V \cap H^2(\Omega)^2) \times (Q \cap H^1(\Omega))$.
Let $(\hat{\ucko}_h, \hat{p}_h) \in V_h \times Q_h$ be the {\rm Stokes} projection of $(\ucko, p)$ by~\eqref{Stokes_projection}.
Then, the following inequalities hold,
\begin{align*}
\n{1}{\SPu - \ucko}, \ \ \n{0}{\SPp - p}, \ \  |\SPp - p|_h & \le \alpha_3 h \n{H^2 \times H^1 }{(\ucko, p)}.
\end{align*}
\end{lemma}
\begin{lemma}[ \cite{LMNT-Peterlin_Oseen_Part_I} ]\label{lem:composite_func}
Under Hypothesis~\ref{hyp:w} and the condition~\eqref{cond:dt_w_Jacobian} the following inequality holds for any $n \in \{0,\ldots,N_T\}$
\begin{align*}
\n{0}{\gcko \circ X_1^n} & \leq (1 + \alpha_4 |\wcko^n |_{1,\infty} \Delta t) \n{0}{\gcko}, && \forall \gcko \in L^2(\Omega)^s,
\end{align*}
where $s= 2$ or $2\times 2$.
\end{lemma}
We present a key lemma in order to deal with the nonlinear terms.
\begin{lemma}\label{lem:vanish_nonlin}
 For $\Ecko \in \mathbb{R}^{2\times 2}$ and $\Dcko \in \mathbb{R}^{2 \times 2}_{sym}$ it holds that
\begin{align}
    (\tr \Dcko)\Dcko : \Ecko - \Ecko \Dcko : \Dcko - \frac{1}{2} (\tr \Ecko ) \Dcko^{\#} : \Dcko = 0.
\label{eq0}
\end{align}
\end{lemma}
%
%
\begin{proof}
The direct calculation yields the result, see also Remark~\ref{rmk:vanish_nonlin}.
\end{proof}
\begin{remark}\label{rmk:vanish_nonlin}
Let $(\ucko, p, \Ccko)$ be a solution of~\eqref{model}.
Multiplying~\eqref{model_ucko} and~\eqref{model_Ccko} by $\ucko$ and $\Ccko/2$, respectively, and adding them, we can obtain an energy inequality on $(\ucko,\Ccko)$ since the term derived from the nonlinear terms of~\eqref{model_ucko} and~\eqref{model_Ccko} vanishes,
\begin{align}
  (\di [(\tr \Ccko)\Ccko], \ucko)+ \frac{1}{2}((\nabla \ucko)\Ccko+\Ccko(\nabla \ucko), \Ccko) = 0.
\label{eq1}
\end{align}
Identity~\eqref{eq1} is proved as follows.
The left-hand side is equal to
\begin{align}
& -( (\tr \Ccko)\Ccko, {\nabla}\ucko)+ ((\nabla \ucko)\Ccko,\Ccko) = (\nabla \ucko, \Ccko \Ccko^T - (\tr \Ccko)\Ccko)
\nonumber\\
& =\int_\Omega \sum_{i,j=1}^{2} \prz{u_i}{x_j} \sum_{k=1}^{2}(C_{ik}C_{jk}-C_{kk}C_{ij})~dx
\nonumber\\
& =\int_\Omega \Bigl( \prz{u_1}{x_1} + \prz{u_2}{x_2} \Bigr) (C_{12}C_{12}-C_{11}C_{22})~dx
= -\frac{1}{2} \bigl( (\di \ucko) \Ccko^{\#}, \Ccko \bigr)
\label{eq2}
\end{align}
Since $\di u=0$, \eqref{eq2} implies \eqref{eq1}.
In the approximate solution $(\ucko_h, p_h, \Ccko_h)$ the exact incompressibility $\di \ucko_h=0$ does not hold.
Hence, \eqref{eq1} is not true, in general, for $(\ucko_h, \Ccko_h)$.
On the other hand, \eqref{eq2} is always valid regardless of the property of $u$.
Therefore, by adding the second term of the right-hand side in (5b), $(\di \ucko_h^n (\Ccko_h^n)^{\#}, \Dcko_h)$, we can obtain the corresponding equation to \eqref{eq1} for $(\ucko_h^n, \Ccko_h^n)$,
\begin{align*}
  -( (\tr \Ccko_h^n) \Ccko_h^n, \nabla \ucko_h^n)+ ((\nabla \ucko_h^n) \Ccko_h^n, \Ccko_h^n)+\frac{1}{2}(\di \ucko_h^n (\Ccko_h^n)^{\#}, \Ccko_h^n) = 0,
\end{align*}
which plays a key role in the following stability analysis.
Identity~\eqref{eq0} is proved similarly to~\eqref{eq2} by replacing $\Ccko$ and $\nabla \ucko$ by $\Dcko$ and $\Ecko$, respectively.
\end{remark}
\begin{remark}
(i)~Lemma~\ref{lem:vanish_nonlin} does not hold in three-dimensional case.
This is the reason why we consider two-dimensional case in this paper.
\smallskip\\
(ii)~By virtue of the term~$( \di \au (\aC)^{\#},\Dcko_h )$ in scheme~\eqref{nonlin_scheme}, we can prove the error estimates for $\varepsilon = 0$, which is an advantage of the nonlinear scheme.
In Part~II, we propose a linear scheme for the model~\eqref{model} and prove error estimates for $\varepsilon > 0$, where the presence of $\Delta \Ccko$ in~\eqref{model_Ccko} is essentially employed.
It is, therefore, not easy to show error estimates of the linear scheme in a similar way for $\varepsilon = 0$.
On the other hand, the linear scheme has an advantage that the proof of the error estimates can be extended to three-dimensional problems.
\end{remark}
%
%
%
%
\begin{lemma}[ \cite{TabTag-2005} ]\label{lem:Gronwall}
Let $a_i$, $i=1, 2$, be non-negative number, $\Delta t$ a positive number, and $\{x^n\}_{n\geq 0}$, $\{y^n\}_{n\geq 1}$ and $\{b^n\}_{n\geq 1}$ non-negative sequences.
Assume $\Delta t \in (0,1/(2 a_0)]$ for $a_0 \neq 0$.
Suppose
\begin{align*}
\ol{D}_{\Delta t} x^n + y^n  \leq a_0 x^n + a_1 x^{n-1} + b^n, \quad \forall n\geq 1.
\end{align*}
Then, it holds that
\begin{align*}
 x^n + \Delta t \sum_{i=1}^n y^i   \leq \exp [ (2a_0+a_1) n \Delta t ] \biggl( x^0  + \Delta t \sum_{i=1}^n b^i \biggr), \quad \forall n\ge 1.
\end{align*}
\end{lemma}
\begin{lemma}[ {\cite[Chap.~II,~Lemma~1.4]{Tem-1984}}, {\cite[Chap.~I,~Lemme~4.3]{Lio-1969}} ]\label{lem:Brouwer}
Let $X$ be a finite dimensional {\rm Hilbert} space with inner product~$(\cdot,\cdot)_X$ and norm~$\|\cdot\|_X$ and let $\mathcal{P}$ be a continuous mapping from $X$ into itself such that $(\mathcal{P}(\xi),\xi)_X > 0$ for $\|\xi\|_X = \rho_0 > 0$.
Then, there exists~$\xi\in X$, $\|\xi\|_X\le \rho_0$, such that $\mathcal{P}(\xi)=0$.
\end{lemma}
%
%
%
\subsection{Proof of Proposition~\ref{prop:existence}}
%
We apply Lemma~\ref{lem:Brouwer} for the proof.
Let $n\in\{1,\ldots, N_T\}$ be a fixed number and $(\ucko_h^{n-1}, \Ccko_h^{n-1}) \in V_h\times W_h$ a pair of given functions.
We set $\mu_0 \defeq {(1-2\Delta t)/2 > 0}$.
We define a finite dimensional inner product space $X \defeq V_h\times Q_h\times W_h$ equipped with the inner product,
\begin{align*}
\bigl( (\ucko_h, p_h, \Ccko_h), (\vcko_h, q_h, \Dcko_h) \bigr)_X & \defeq \fz{1}{\Delta t} (\ucko_h, \vcko_h) + 4 \nu \bigl( \D{\ucko_h}, \D{\vcko_h} \bigr) \\
& \quad + 2 \delta_0 \sum_{K\in \mathcal{T}_h} h_K^2 (p_h, q_h)_K + \fz{\mu_0}{\Delta t} (\Ccko_h, \Dcko_h) + \varepsilon (\nabla\Ccko_h, \nabla\Dcko_h),
\end{align*}
which induces the norm~$\|\cdot\|_X$ for any $\varepsilon \ge 0$.
Let $\mathcal{P}: V_h\times Q_h\times W_h \to V_h\times Q_h\times W_h$ be a mapping defined by
\begin{align}
\bigl( \mathcal{P} (\ucko_h, p_h, \Ccko_h), (\vcko_h, q_h, \Dcko_h) \bigr)_X
& = \biggl( \frac{\ucko_h-\aou\circ X_1^n}{\Delta t}, \vcko_h \biggr) + \mathcal{A}_h \bigl( ( \ucko_h, p_h ), (\vcko_h, -q_h) \bigr) + \bigl( (\tr\Ccko_h)\Ccko_h, \nabla \vcko_h \bigr) \notag\\
& \quad - ( \af,\vcko_h ) + \fz{1}{2} \biggl( \frac{\Ccko_h-\aoC \circ X_1^n}{\Delta t}, \Dcko_h \biggr) + \fz{\varepsilon}{2} a_{c}( \Ccko_h, \Dcko_h ) - \bigl( (\nabla\ucko_h)\Ccko_h, \Dcko_h \bigr) \notag\\
& \quad - \fz{1}{2} \bigl( (\di \ucko_h) \Ccko_h^{\#},\Dcko_h \bigr) + \fz{1}{2} \bigl( (\tr\Ccko_h)^2 \Ccko_h, \Dcko_h \bigr) - \fz{1}{2}\bigl( (\tr\Ccko_h)\Icko,\Dcko_h \bigr) \notag\\
& \quad - \fz{1}{2} ( \aF, \Dcko_h ),
\qquad\qquad \forall (\ucko_h, p_h, \Ccko_h), (\vcko_h,q_h,\Dcko_h) \in V_h\times Q_h \times W_h.
\label{nonlin_scheme_1}
\end{align}
Obviously $\mathcal{P}$ is continuous.
Substituting~$(\ucko_h, p_h, \Ccko_h)$ into $(\vcko_h,q_h,\Dcko_h)$ in~\eqref{nonlin_scheme_1} and using the inequality $\|\tr\Ccko_h\|_0 \le \sqrt{2} \|\Ccko_h\|_0$, we have
\begin{align*}
& \bigl( \mathcal{P} (\ucko_h, p_h, \Ccko_h), (\ucko_h, p_h, \Ccko_h) \bigr)_X \\
& = \biggl( \frac{\ucko_h-\aou\circ X_1^n}{\Delta t}, \ucko_h \biggr) + 2\nu \|\D{\ucko_h}\|_0^2 + \delta_0 |p_h|_h^2 - ( \af,\ucko_h ) \\
& \quad + \fz{1}{2} \biggl( \frac{\Ccko_h-\aoC \circ X_1^n}{\Delta t}, \Ccko_h \biggr) + \fz{\varepsilon}{2} |\Ccko_h|_1^2 + \fz{1}{2} \|(\tr\Ccko_h) \Ccko_h\|_0^2
- \fz{1}{2} \|\tr\Ccko_h\|_0^2 - \fz{1}{2} ( \aF, \Ccko_h ) \\
& \ge \fz{1}{\Delta t}\bigl( \|\ucko_h\|_0^2 - \|\aou\circ X_1^n\|_0\|\ucko_h\|_0) + 2\nu \|\D{\ucko_h}\|_0^2 + \delta_0 |p_h|_h^2 - \|\af\|_0 \|\ucko_h\|_0 \\
& \quad + \fz{1}{2\Delta t} \bigl( \|\Ccko_h\|_0^2 - \|\aoC \circ X_1^n\|_0 \|\Ccko_h\|_0 \bigr) + \fz{\varepsilon}{2} |\Ccko_h|_1^2 - \|\Ccko_h\|_0^2 - \fz{1}{2} \|\aF\|_0 \|\Ccko_h\|_0 \qquad \mbox{(by Schwarz' inequality)} \\
& \ge \fz{1}{2\Delta t}\Bigl\{ 2\|\ucko_h\|_0^2 - \beta_0 \|\ucko_h\|_0^2 - \fz{1}{\beta_0}\|\aou\circ X_1^n\|_0^2 + \|\Ccko_h\|_0^2 - \beta_1 \|\Ccko_h\|_0^2 - \fz{1}{4\beta_1} \|\aoC \circ X_1^n\|_0^2 \Bigr\} \\
& \quad + 2 \nu \|\D{\ucko_h}\|_0^2 + \delta_0 |p_h|_h^2 - \fz{\beta_2}{2\Delta t} \|\ucko_h\|_0^2 - \fz{\Delta t}{2\beta_2} \|\af\|_0^2 + \fz{\varepsilon}{2} |\Ccko_h|_1^2 - \|\Ccko_h\|_0^2 - \fz{\beta_3}{2\Delta t} \|\Ccko_h\|_0^2 - \fz{\Delta t}{8 \beta_3} \|\aF\|_0^2 \\
& \qquad\qquad\qquad\qquad\qquad\qquad\qquad\qquad\qquad\qquad\qquad\qquad\qquad\qquad\qquad\qquad\qquad\qquad\quad
\mbox{(by $ab \le \fz{\beta}{2}a^2+\fz{1}{2\beta}b^2$)} \\
& \ge \fz{1}{2\Delta t}\Bigl\{ (2 - \beta_0 - \beta_2) \|\ucko_h\|_0^2 + (1- \beta_1 - 2\Delta t - \beta_3) \|\Ccko_h\|_0^2 \Bigr\} + 2 \nu \|\D{\ucko_h}\|_0^2 + \delta_0 |p_h|_h^2 \\
& \quad + \fz{\varepsilon}{2} |\Ccko_h|_1^2 - \fz{1}{2\beta_0\Delta t}\|\aou\circ X_1^n\|_0^2 - \fz{1}{8\beta_1\Delta t} \|\aoC \circ X_1^n\|_0^2 - \fz{\Delta t}{2\beta_2} \|\af\|_0^2 - \fz{\Delta t}{8 \beta_3} \|\aF\|_0^2 \qquad\quad \mbox{(by Lemma~\ref{lem:composite_func})}
\end{align*}
for any $\beta_i > 0$.
Choosing $\beta_0 = \beta_2 = 1/2$ and $\beta_1 = \beta_3 = \mu_0/2$, we get
\begin{align}
\bigl( \mathcal{P} (\ucko_h, p_h, \Ccko_h), (\ucko_h, p_h, \Ccko_h) \bigr)_X
& \ge \fz{1}{2}\biggl[ \biggl\{ \fz{1}{\Delta t}\|\ucko_h\|_0^2 + 4 \nu \|\D{\ucko_h}\|_0^2 + 2 \delta_0 |p_h|_h^2 + \fz{\mu_0}{\Delta t}\|\Ccko_h\|_0^2 + \varepsilon |\Ccko_h|_1^2 \biggr\} \notag\\
& \qquad -
\biggl\{ \fz{2\|\aou\circ X_1^n\|_0^2}{\Delta t} + \fz{\|\aoC \circ X_1^n\|_0^2}{2\mu_0\Delta t} + 2\Delta t \|\af\|_0^2 + \fz{\Delta t \|\aF\|_0^2}{2\mu_0} \biggr\} \biggr] \notag\\
& = \fz{1}{2} \Bigl[ \| (\ucko_h, p_h, \Ccko_h) \|_X^2 - \beta_\ast^2 \Bigr], \notag
\end{align}
where
\[
\beta_\ast := \biggl\{ \fz{2\|\aou\circ X_1^n\|_0^2}{\Delta t} + \fz{\|\aoC \circ X_1^n\|_0^2}{2\mu_0\Delta t} + 2\Delta t \|\af\|_0^2 + \fz{\Delta t \|\aF\|_0^2}{2\mu_0} \biggr\}^{1/2}.
\]
The right-hand side is, therefore, positive on the sphere of radius~$\rho_0 = \beta_\ast + 1$.
From Lemma~\ref{lem:Brouwer} there exists an element~$(\ucko_h, p_h, \Ccko_h)\in V_h\times Q_h\times W_h$ such that $\mathcal{P} (\ucko_h, p_h, \Ccko_h)=0$, which is nothing but a solution of equations~\eqref{nonlin_scheme}.
\qed
%
%
%
%
%
%
%
%
%
%
%
%
\subsection{A system of equations for the error and the estimate of remainder terms}
In this subsection we prepare a system of equations for the error and a lemma for the estimate of remainder terms in the system before starting the proof of Theorem~\ref{thm:error_estimates}.
\par
Let $(\hat{\ucko}_h, \hat{p}_h)(t) \defeq \Pi_h^{\rm S} (\ucko, p)(t) \in V_h \times Q_h$ and $\check{\Ccko}_h(t) \defeq \Pi_h \Ccko (t) \in W_h$ for $t\in [0,T]$ and let
\begin{align*}
\ecko_h^n & \defeq \ucko_h^n-\hat{\ucko}_h^n,
&
\epsilon_h^n &\defeq p_h^n-\hat{p}_h^n,
&
\Ecko_h^n & \defeq \Ccko_h^n-\check{\Ccko}_h^n,
&
\etacko (t) & \defeq (\ucko-\hat{\ucko}_h)(t),
&
\Xicko (t) & \defeq (\Ccko-\check{\Ccko}_h)(t).
\end{align*}
Then, from~\eqref{nonlin_scheme}, \eqref{Stokes_projection} and~\eqref{weak_formulation}, we have for $n\ge 1$
\begin{subequations}\label{eqns:error}
\begin{align}
\biggl( \fz{\ecko_h^n - \ecko_h^{n-1} \circ X_1^n}{\Delta t}, \vcko_h \biggr) & + \mathcal{A}_h \bigl( (\ecko_h^n,\epsilon_h^n), (\vcko_h,q_h) \bigr)
= - \bigl( (\tr\Ecko_h^n)\Ecko_h^n, \nabla \vcko_h \bigr) + {}_{V_h^\prime}\lA \rcko_h^n, \vcko_h\rA_{V_h},\qquad
\label{eq:error_up} \\
\biggl( \fz{\Ecko_h^n - \Ecko_h^{n-1} \circ X_1^n}{\Delta t}, \Dcko_h \biggr) & + \varepsilon a_c (\Ecko_h^n, \Dcko_h) = 2 \bigl( (\nabla \ecko_h^n)\Ecko_h^n, \Dcko_h \bigr) + \bigl( (\di \ecko_h^n) (\Ecko_h^n)^{\#}, \Dcko_h \bigr) + {}_{W_h^\prime}\lA \Rcko_h^n, \Dcko_h\rA_{W_h},
\label{eq:error_C}\\
& \qquad\qquad\qquad\qquad\qquad\qquad\qquad\qquad\qquad\qquad \forall (\vcko_h, q_h, \Dcko_h) \in V_h\times Q_h \times W_h, \notag
\end{align}
\end{subequations}
where
\begin{align*}
\rcko_h^n & \defeq \sum_{i=1}^4 \rcko_{hi}^n \in V_h^\prime, \qquad \Rcko_h^n \defeq \sum_{i=1}^{11} \Rcko_{hi}^n \in W_h^\prime, \\
( \rcko_{h1}^n, \vcko_h ) & \defeq \( \Dt{\ucko^n} - \fz{\ucko^n - \ucko^{n-1} \circ X_1^n}{\Delta t}, \vcko_h\), \\
( \rcko_{h2}^n, \vcko_h ) & \defeq \fz{1}{\Delta t}\( \etacko^n - \etacko^{n-1} \circ X_1^n, \vcko_h \), \\
{}_{V_h^\prime}\lA \rcko_{h3}^n, \vcko_h \rA_{V_h} & \defeq -\bigl( (\tr \check{\Ccko}_h^n )\Ecko_h^n + (\tr\Ecko_h^n) \check{\Ccko}_h^n, \nabla\vcko_h \bigr), \\
{}_{V_h^\prime}\lA \rcko_{h4}^n, \vcko_h \rA_{V_h} & \defeq \bigl( (\tr \check{\Ccko}_h^n )\Xicko^n + (\tr\Xicko^n)\Ccko^n, \nabla\vcko_h \bigr), \\
( \Rcko_{h1}^n, \Dcko_h ) & \defeq \( \Dt{\Ccko^n} - \fz{\Ccko^n - \Ccko^{n-1} \circ X_1^n}{\Delta t}, \Dcko_h\), \\
( \Rcko_{h2}^n, \Dcko_h ) & \defeq \fz{1}{\Delta t} \( \Xicko^n - \Xicko^{n-1} \circ X_1^n, \Dcko_h\),\\
{}_{W_h^\prime}\lA \Rcko_{h3}^n, \Dcko_h \rA_{W_h} & \defeq \varepsilon a_c(\Xicko^n, \Dcko_h), \\
( \Rcko_{h4}^n, \Dcko_h ) & \defeq 2\bigl( (\nabla \SPu^n) \Ecko_h^n + (\nabla \ecko_h^n) \check{\Ccko}_h^n, \Dcko_h\bigr), \\
( \Rcko_{h5}^n, \Dcko_h ) & \defeq -2\bigl( (\nabla \SPu^n) \Xicko^n + (\nabla \etacko^n) \Ccko^n, \Dcko_h\bigr), \\
( \Rcko_{h6}^n, \Dcko_h ) & \defeq \bigl( (\di \SPu^n) (\Ecko_h^n)^{\#} + (\di \ecko_h^n) ( \check{\Ccko}_h^n )^{\#}, \Dcko_h \bigr),\\
( \Rcko_{h7}^n, \Dcko_h ) & \defeq -\bigl( (\di \SPu^n) (\Xicko^n)^{\#} + (\di \etacko^n) (\Ccko^n)^{\#}, \Dcko_h \bigr),\\
( \Rcko_{h8}^n, \Dcko_h ) & \defeq -\bigl( [\tr (\Ecko_h^n+ \check{\Ccko}_h^n )]^2 \Ecko_h^n, \Dcko_h \bigr),\\
( \Rcko_{h9}^n, \Dcko_h ) & \defeq -\bigl( [\tr (\Ecko_h^n+2 \check{\Ccko}_h^n )] (\tr\Ecko_h^n) \check{\Ccko}_h^n, \Dcko_h\bigr),\\
( \Rcko_{h10}^n, \Dcko_h ) & \defeq \bigl( (\tr \check{\Ccko}_h^n )^2 \Xicko^n + [\tr(\Ccko^n + \check{\Ccko}_h^n )] (\tr\Xicko^n)\Ccko^n, \Dcko_h\bigr),\\
( \Rcko_{h11}^n, \Dcko_h ) & \defeq \( [\tr(\Ecko_h^n-\Xicko^n)]\Icko, \Dcko_h\).
\end{align*}
\par
The remainder terms are evaluated by the next lemma.
\begin{lemma}\label{lem:estimates_r_R}
Suppose Hypotheses~\ref{hyp:w} and~\ref{hyp:regularity} hold.
Let $n\in \{1,\ldots,N_T\}$ be any fixed number.
Then, under the condition~\eqref{cond:dt_w_Jacobian} it holds that
\begin{subequations}
\begin{align}
\| \rcko_{h1}^n \|_0 & \le c_w\sqrt{\Delta t} \|\ucko\|_{Z^2(t^{n-1},t^n)},
\label{ieq:r1}
\\
\| \rcko_{h2}^n \|_0 & \le \fz{c_w h}{\sqrt{\Delta t}} \| (\ucko, p) \|_{H^1(t^{n-1},t^n; H^2\times H^1)},
\label{ieq:r2}
\\
\| \rcko_{h3}^n \|_{-1} & \le c_s \|\Ecko_h^n\|_0,
\label{ieq:r3}
\\
\| \rcko_{h4}^n \|_{-1} & \le c_s h,
\label{ieq:r4}
\\
\| \Rcko_{h1}^n \|_0 & \le c_w \sqrt{\Delta t} \|\Ccko\|_{Z^2(t^{n-1},t^n)},
\label{ieq:R1}
\\
\| \Rcko_{h2}^n \|_0 & \le \fz{c_w h}{\sqrt{\Delta t}} \| \Ccko \|_{H^1(t^{n-1},t^n; H^1) \cap L^2(t^{n-1},t^n; H^2)},
\label{ieq:R2}
\\
{\!\!\phantom{\Bigl|}}_{W_h^\prime}\Bigl\lA \Rcko_{h3}^n, \fz{1}{2}\Ecko_h^n \Bigr\rA_{W_h} & \le \fz{\varepsilon}{4} |\Ecko_h^n|_1^2 + c_s h^2,
\label{ieq:R3}
\\
\| \Rcko_{h4}^n \|_0 & \le c_s (\|\ecko_h^n\|_1+\|\Ecko_h^n\|_0),
\label{ieq:R4}
\\
\| \Rcko_{h5}^n \|_0 & \le c_s h,
\label{ieq:R5}
\\
\| \Rcko_{h6}^n \|_0 & \le c_s \bigl( \|\ecko_h^n\|_1 + \|\Ecko_h^n\|_0\bigr),
\label{ieq:R6}
\\
\| \Rcko_{h7}^n \|_0 & \le c_s h,
\label{ieq:R7}
\\
\Bigl( \Rcko_{h8}^n, \fz{1}{2}\Ecko_h^n \Bigr) & \le -\fz{3}{8}\|(\tr\Ecko_h^n) \Ecko_h^n\|_0^2 + c_s \|\Ecko_h^n\|_0^2,
\label{ieq:R8}
\\
\Bigl( \Rcko_{h9}^n, \fz{1}{2}\Ecko_h^n \Bigr) & \le \fz{1}{8} \| (\tr\Ecko_h^n) \Ecko_h^n\|_0^2 + c_s\|\Ecko_h^n\|_0^2,
\label{ieq:R9}
\\
\| \Rcko_{h10}^n \|_0 & \le c_s h,
\label{ieq:R10}
\\
\|\Rcko_{h11}^n\|_0 & \le c_s ( \|\Ecko_h^n\|_0 + h),
\label{ieq:R11}
\end{align}
\end{subequations}
where $c_w$ and $c_s$ are the constants given in the beginning of Section~\ref{sec:main_result}.
\end{lemma}
\begin{proof}
Let $t(s)\defeq t^{n-1} +s\Delta t~(s\in [0,1])$ and~$y(x, s)\defeq x-(1-s)\wcko^n(x)\Delta t$.
\par
We prove~\eqref{ieq:r1}.
We have that
\begin{align}
\rcko_{h1}^n(x)
& = \Bigl\{ \Bigl( \prz{}{t} + \wcko^n(x)\cdot\nabla \Bigr) \ucko \Bigr\} (x, t^n) - \fz{1}{\Delta t} \Bigl[ \ucko \bigl( y(x,s), t(s) \bigr) \Bigr]_{s=0}^1 \notag\\
& = \Bigl\{ \Bigl( \prz{}{t} + \wcko^n(x)\cdot\nabla \Bigr) \ucko \Bigr\} (x, t^n) - \int_0^1 \Bigl\{ \Bigl( \prz{}{t} + \wcko^n(x)\cdot\nabla \Bigr) \ucko \Bigr\} \bigl( y(x,s), t(s) \bigr) ds \notag\\
& = \Delta t \int_0^1ds \int_{s}^1 \Bigl\{ \Bigl( \prz{}{t} + \wcko^n(x)\cdot\nabla \Bigr)^2 \ucko \Bigr\} \bigl( y(x,s_1), t(s_1) \bigr) ds_1 \notag\\
& = \Delta t \int_0^1 s_1 \Bigl\{ \Bigl( \prz{}{t} + \wcko^n(x)\cdot\nabla \Bigr)^2 \ucko \Bigr\} \bigl( y(x,s_1), t(s_1) \bigr) ds_1, \notag
\intertext{which implies}
\| \rcko_{h1}^n \|_0
& \le \Delta t \int_0^1 s_1 \Bigl\| \Bigl\{ \Bigl( \prz{}{t} + \wcko^n(\cdot)\cdot\nabla \Bigr)^2 \ucko \Bigr\} \bigl( y(\cdot,s_1), t(s_1) \bigr) \Bigr\|_0 ds_1
\le c_w\sqrt{\Delta t} \| \ucko \|_{Z^2(t^{n-1},t^n)}, \notag
\end{align}
where for the last inequality we have changed the variable from $x$ to $y$ and used the evaluation $\det (\pz y(x, s_1)/\pz x) \ge 1/2~(\forall s_1\in [0,1])$ from Proposition~\ref{prop:RT_TU}-(ii).
\par
We prove~\eqref{ieq:r2}.
From the equalities,
\begin{align*}
\rcko_{h2}^n
& = \fz{1}{\Delta t} \Bigl[ \etacko \bigl( y(\cdot,s), t(s) \bigr) \Bigr]_{s=0}^1 = \int_0^1 \Bigl\{ \Bigl( \prz{}{t} + \wcko^n(\cdot)\cdot\nabla \Bigr) \etacko \Bigr\} \bigl( y(\cdot,s), t(s) \bigr) ds,
\end{align*}
we have
\begin{align*}
\| \rcko_{h2}^n \|_0
& \le \int_0^1 \Bigl\| \Bigl\{ \Bigl( \prz{}{t} + \wcko^n(\cdot)\cdot\nabla \Bigr) \etacko \Bigr\} \bigl( y(\cdot, s), t(s) \bigr) \Bigr\|_0 ds
\le \int_0^1 \Bigl( \Bigl\|\prz{\etacko}{t} \bigl( y(\cdot, s), t(s) \bigr) \Bigr\|_0 + c_w \bigl\| \nabla \etacko \bigl( y(\cdot, s), t(s) \bigr) \bigr\|_0 \Bigr) ds \\
& \le \sqrt{2} \int_0^1 \Bigl\{ \Bigl\|\prz{\etacko}{t} \bigl( \cdot, t(s) \bigr) \Bigr\|_0 + c_w \bigl\| \nabla \etacko \bigl( \cdot, t(s) \bigr) \bigr\|_0 \Bigr\} ds
\le \sqrt{ \fz{2}{\Delta t} } \Bigl( \Bigl\| \prz{\etacko}{t} \Bigr\|_{L^2(t^{n-1},t^n; L^2)} + c_w \bigl\| \nabla \etacko \bigr\|_{L^2(t^{n-1},t^n; L^2)} \Bigr) \\
& \le \sqrt{ \fz{2}{\Delta t} } \alpha_{31}h (1+c_w) \| (\ucko,p) \|_{H^1(t^{n-1},t^n; H^2\times H^1)}
\le \fz{c_w^\prime h}{\sqrt{\Delta t}} \| (\ucko,p) \|_{H^1(t^{n-1},t^n; H^2\times H^1)},
\end{align*}
which leads to~\eqref{ieq:r2}, where Proposition~\ref{prop:RT_TU}-(ii) has been used for the third inequality.
\par
From Lemmas~\ref{lem:interpolation} and~\ref{lem:estimates_Stokes_projection}, \eqref{ieq:r3} and~\eqref{ieq:r4} are obtained as follows:
\begin{align*}
\| \rcko_{h3}^n \|_{-1}
& \le \| (\tr \check{\Ccko}_h^n)\Ecko_h^n + (\tr\Ecko_h^n) \check{\Ccko}_h^n \|_0
\le c \| \check{\Ccko}_h^n \|_{0,\infty} \|\Ecko_h^n\|_0
\le c \|\Ccko\|_{C(L^\infty)} \|\Ecko_h^n\|_0
\le c_s \|\Ecko_h^n\|_0, \\
\| \rcko_{h4}^n \|_{-1}
& \le \| (\tr \check{\Ccko}_h^n )\Xicko^n + (\tr\Xicko^n)\Ccko^n \|_0
\le c \| \check{\Ccko}_h^n \|_{0,\infty} \|\Xicko_h^n\|_0
\le c \|\Ccko\|_{C(L^\infty)} \alpha_{21} h \|\Ccko\|_{C(H^1)}
\le c_s h.
\end{align*}
\par
The estimate~\eqref{ieq:R1} is obtained by replacing $\ucko$ with $\Ccko$ in the proof of~\eqref{ieq:r1}.
\par
We prove~\eqref{ieq:R2}.
Replacing~$\etacko$ with~$\Xicko$ in the estimate of~$\| \rcko_{h2}^n \|_0$ above, we have
\begin{align*}
\| \Rcko_{h2}^n \|_0
& \le \sqrt{ \fz{2}{\Delta t} } \Bigl( \Bigl\| \prz{\Xicko}{t} \Bigr\|_{L^2(t^{n-1},t^n; L^2)} + c_w \bigl\| \nabla \Xicko \bigr\|_{L^2(t^{n-1},t^n; L^2)} \Bigr) \\
& \le \sqrt{ \fz{2}{\Delta t} } h \Bigl( \alpha_{21} \| \Ccko \|_{H^1(t^{n-1},t^n; H^1)} + c_w \alpha_{22} \| \Ccko \|_{L^2(t^{n-1},t^n; H^2)} \Bigr) \\
&
\le \fz{c_w^\prime h}{\sqrt{\Delta t}} \| \Ccko \|_{H^1(t^{n-1},t^n; H^1) \cap L^2(t^{n-1},t^n; H^2)},
\end{align*}
which implies~\eqref{ieq:R2}.
\par
The estimate~\eqref{ieq:R3} is obtained from
\begin{align*}
{\!\!\phantom{\Bigl|}}_{W_h^\prime}\Bigl\lA \Rcko_{h3}^n, \fz{1}{2}\Ecko_h^n \Bigr\rA_{W_h}
& \le \fz{\varepsilon}{2} |\Xicko^n|_1 |\Ecko_h^n|_1
\le \fz{\varepsilon}{4} (|\Ecko_h^n|_1^2 + |\Xicko^n|_1^2)
& \mbox{(by $ab \le (a^2+b^2)/2$)}\\
&\le \fz{\varepsilon}{4} ( |\Ecko_h^n|_1^2 + \alpha_3^2 h^2 \|\Ccko\|_{C(H^2)}^2)
\le \fz{\varepsilon}{4} |\Ecko_h^n|_1^2 + c_s h^2.
\end{align*}
\par
In order to prove estimates~\eqref{ieq:R4}--\eqref{ieq:R7} we prepare the boundedness of~$\| \nabla \SPu^n \|_{0,\infty}$.
Let $\check{\ucko}_h(t) \defeq (\Pi_h\ucko) (t)$ for $t\in [0,T]$.
We have
\begin{align}
\| \nabla \SPu^n \|_{0,\infty}
& \le \| \SPu^n \|_{1,\infty}
\le \| \SPu^n - \check{\ucko}_h^n\|_{1,\infty} + \| \check{\ucko}_h^n \|_{1,\infty}
\le \alpha_{25} h^{-1} \| \SPu^n - \check{\ucko}_h^n\|_1 + \alpha_{20} \| \ucko^n \|_{1,\infty} \notag\\
& \le \alpha_{25} h^{-1} \bigl( \| \SPu^n - \ucko^n\|_1 + \| \ucko^n - \check{\ucko}_h^n\|_1 \bigr) + \alpha_{20} \| \ucko^n \|_{1,\infty} \notag\\
& \le \alpha_{25} h^{-1} \bigl( \alpha_3 h \| (\ucko, p)^n \|_{H^2\times H^1} + \alpha_{22} h \| \ucko^n \|_2 \bigr) + \alpha_{20} \| \ucko^n \|_{1,\infty} \notag\\
& \le \alpha_{25} ( \alpha_{22} + \alpha_3 ) \| (\ucko, p) \|_{C(H^2\times H^1)} + \alpha_{20} \| \ucko \|_{C(W^{1,\infty})} \le c_s.
\label{ieq:bound_u_hat_1inf}
\end{align}
\par
We prove~\eqref{ieq:R4}--\eqref{ieq:R7} by using~\eqref{ieq:bound_u_hat_1inf} and~\eqref{ieq:Pi_h_0inf} as follows.
\begin{align*}
\| \Rcko_{h4}^n \|_0
& \le 2 ( \| (\nabla \SPu^n) \Ecko_h^n \|_0 + \| (\nabla \ecko_h^n) \check{\Ccko}_h^n \|_0)
\le  c ( c_s \| \Ecko_h^n \|_0 + \|\Ccko\|_{C(L^\infty)} \| \nabla \ecko_h^n \|_0 )
\le  c_s^\prime ( \| \ecko_h^n \|_1 + \| \Ecko_h^n \|_0 ), \\
\|\Rcko_{h5}^n\|_0 & \le 2 ( \| (\nabla \SPu^n) \Xicko^n\|_0 + \|(\nabla \etacko^n) \Ccko^n\|_0 )
\le c ( \| \nabla \SPu^n \|_{0,\infty} \| \Xicko^n\|_0 + \|\Ccko\|_{C(L^\infty)} \|\nabla \etacko^n\|_0 ) \\
& \le c_s ( \| \Xicko^n\|_0 + \|\etacko^n\|_1 )
\le c_s h ( \alpha_{21} \| \Ccko\|_{C(H^1)} + \alpha_3 \|(\ucko, p)\|_{C(H^2\times H^1)} )
\le c_s^\prime h, \\
%
\|\Rcko_{h6}^n\|_0 & \le \|\nabla \SPu^n\|_{0,\infty} \|\Ecko_h^n\|_0 + \|\check{\Ccko}_h^n\|_{0,\infty} \|\ecko_h^n\|_1
\le c_s \|\Ecko_h^n\|_0 + \|\Ccko\|_{C(L^\infty)} \|\ecko_h^n\|_1
\le c_s^\prime (\|\Ecko_h^n\|_0 + \|\ecko_h^n\|_1), \\
%
\| \Rcko_{h7}^n \|_0
& \le \|\nabla \SPu^n\|_{0,\infty} \| \Xicko^n \|_0 + \| \Ccko^n \|_{0,\infty} \|\etacko^n\|_1
\le c_s ( \| \Xicko^n \|_0 + \|\etacko^n\|_1 ) \\
& \le c_s h ( \alpha_{21} \| \Ccko\|_{C(H^1)} + \alpha_3 \|(\ucko, p)\|_{C(H^2\times H^1)} )
\le c_s^\prime h.
\end{align*}
\par
The remainder estimates~\eqref{ieq:R8}--\eqref{ieq:R11} are obtained from
\begin{align*}
\Bigl( \Rcko_{h8}^n, \fz{1}{2}\Ecko_h^n \Bigr)
& = - \fz{1}{2} \bigl( [ (\tr\Ecko_h^n)^2 + 2(\tr \Ecko_h^n)(\tr \check{\Ccko}_h^n )+(\tr \check{\Ccko}_h^n )^2] \Ecko_h^n, \Ecko_h^n \bigr) \\
& \le - \fz{1}{2} \| (\tr\Ecko_h^n) \Ecko_h^n \|_0^2 - \bigl( (\tr \Ecko_h^n) \Ecko_h^n, (\tr \check{\Ccko}_h^n ) \Ecko_h^n \bigr)
\\
& \le - \fz{1}{2} \| (\tr\Ecko_h^n) \Ecko_h^n \|_0^2 + \fz{1}{8}\| (\tr \Ecko_h^n) \Ecko_h^n\|_0^2 + 2\| (\tr \check{\Ccko}_h^n ) \Ecko_h^n \|_0^2 \\
& \le -\fz{3}{8}\|(\tr\Ecko_h^n) \Ecko_h^n\|_0^2 + c \|\Ccko\|_{C(L^\infty)}^2 \|\Ecko_h^n\|_0^2
\le -\fz{3}{8}\|(\tr\Ecko_h^n) \Ecko_h^n\|_0^2 + c_s \|\Ecko_h^n\|_0^2
& \mbox{(by~\eqref{ieq:Pi_h_0inf}),} \\
\Bigl( \Rcko_{h9}^n, \fz{1}{2}\Ecko_h^n \Bigr)
& = -\fz{1}{2}\bigl( (\tr\Ecko_h^n) \check{\Ccko}_h^n, (\tr\Ecko_h^n)\Ecko_h^n \bigr)
-\bigl( (\tr \check{\Ccko}_h^n ) (\tr\Ecko_h^n) \check{\Ccko}_h^n, \Ecko_h^n \bigr) \\
& \le \fz{1}{8} \| (\tr\Ecko_h^n) \Ecko_h^n\|_0^2 + c \|\Ccko\|_{C(L^\infty)}^2 \|\Ecko_h^n\|_0^2
\le \fz{1}{8} \| (\tr\Ecko_h^n) \Ecko_h^n\|_0^2 + c_s \|\Ecko_h^n\|_0^2, \\
%
\| \Rcko_{h10}^n \|_0
& \le c \bigl[ \|\check{\Ccko}_h^n\|_{0,\infty}^2  + \|\Ccko^n\|_{0,\infty} (\|\Ccko^n\|_{0,\infty} + \|\check{\Ccko}_h^n\|_{0,\infty}) \bigr] \|\Xicko^n\|_0 \\
& \le c^\prime \|\Ccko\|_{C(L^\infty)} \bigl( 1 + \|\Ccko\|_{C(L^\infty)} \bigr) \|\Xicko^n\|_0
& \mbox{(by~\eqref{ieq:Pi_h_0inf})} \\
& \le c_s \|\Xicko^n\|_0
\le c_s \alpha_{21} h \|\Ccko^n\|_1
\le c_s^\prime h, \\
%
\| \Rcko_{h11}^n \|_0
& \le c ( \| \Ecko_h^n \|_0 + \|\Xicko^n\|_0 )
\le c ( \| \Ecko_h^n \|_0 + \alpha_{21}h \|\Ccko\|_{C(H^1)} )
\le c_s ( \| \Ecko_h^n \|_0 + h ).
&& \qedhere
\end{align*}
\end{proof}
%
%
%
%
%
%
%
%
%
\subsection{Proof of Theorem~\ref{thm:error_estimates}}
%
The constant $h_0$ can be chosen arbitrarily, say, $h_0=1$.
We fix $\Delta t_0$ by
\begin{align}
\Delta t_0 =\min \left\{  \fz{1}{ 4|\wcko|_{C(W^{1,\infty})}}, \fz{1}{2 c_s} \right\} ,
\label{cond:t0}
\end{align}
where $c_s$ is the constant appearing in \eqref{ieq:error_substitution_1} below.
We consider any pair $(h, \Delta t)$ satisfying \eqref{cond:ht} and
any solution~$(\ucko_h, p_h, \Ccko_h)$ of scheme~\eqref{nonlin_scheme} with \eqref{cond:if}.
We return to the argument in the previous subsection.
Substituting~$(\ecko_h^n, -\epsilon_h^n, \fz{1}{2} \Ecko_h^n)$ into~$(\vcko_h, q_h, \Dcko_h)$ in~\eqref{eqns:error} and noting that
\begin{align}
\( \fz{\ecko_h^n - \ecko_h^{n-1} \circ X_1^n}{\Delta t}, \ecko_h^n \)
& \ge \fz{1}{2\Delta t} \Bigl[  \|\ecko_h^n\|_0^2 - (1+\alpha_4 | \wcko^n|_{1,\infty} \Delta t)^2\|\ecko_h^{n-1}\|_0^2 \Bigr]
\ge \ol{D}_{\Delta t} \Bigl( \fz{1}{2} \|\ecko_h^n\|_0^2 \Bigr) - c_w \|\ecko_h^{n-1}\|_0^2 \label{ieq:e_h} \\
& \qquad\qquad\qquad\qquad\qquad\qquad\qquad\qquad
\mbox{(by $(b-a)b \ge (b^2-a^2)/2$ and Lemma~\ref{lem:composite_func})}, \notag\\
\mathcal{A}_h\bigl( (\ecko_h^n, \epsilon_h^n), (\ecko_h^n, -\epsilon_h^n) \bigr) & = 2\nu \|\D{\ecko_h^n}\|_0^2 + \delta_0 |\epsilon_h^n|_h^2 \ge \fz{2\nu}{\alpha_1^2} \|\ecko_h^n\|_1^2 + \delta_0 |\epsilon_h^n|_h^2 \qquad\qquad\qquad\qquad\, \mbox{(by Lemma~\ref{lem:Korn})}, \notag\\
{}_{V_h^\prime}\lA \rcko_h^n, \ecko_h^n \rA_{V_h} & \le \|\rcko_h^n \|_{-1} \|\ecko_h^n\|_1
\le \fz{\alpha_1^2}{4\nu} \| \rcko_h^n \|_{-1}^2 + \fz{\nu}{\alpha_1^2}\|\ecko_h^n\|_1^2 \qquad\qquad  \mbox{(by $ab \le (\beta/4)a^2+(1/\beta)b^2$)}, \notag\\
\( \fz{\Ecko_h^n - \Ecko_h^{n-1} \circ X_1^n}{\Delta t}, \fz{1}{2} \Ecko_h^n\) & \ge \ol{D}_{\Delta t} \Bigl( \fz{1}{4} \|\Ecko_h^n\|_0^2 \Bigr) - c_w \|\Ecko_h^{n-1}\|_0^2 \qquad\qquad\qquad\qquad\qquad\qquad\qquad\qquad\, \mbox{(cf. \eqref{ieq:e_h})}, \notag\\
\varepsilon a_c \Bigl( \Ecko_h^n, \fz{1}{2} \Ecko_h^n \Bigr) & = \fz{\varepsilon}{2} |\Ecko_h^n|_1^2, \notag
\end{align}
and Lemma~\ref{lem:vanish_nonlin}, we have
\begin{align}
\ol{D}_{\Delta t} \Bigl( \fz{1}{2} \|\ecko_h^n\|_0^2 + \fz{1}{4} \|\Ecko_h^n\|_0^2 \Bigr) & + \fz{\nu}{\alpha_1^2}\|\ecko_h^n\|_1^2 + \delta_0 |\epsilon_h^n|_h^2 + \fz{\varepsilon}{2} |\Ecko_h^n|_1^2
\notag\\
& \le c_w (\|\ecko_h^{n-1}\|_0^2 + \|\Ecko_h^{n-1}\|_0^2 ) + \fz{\alpha_1^2}{4\nu} \| \rcko_h^n \|_{-1}^2 + {\!\!\phantom{\Bigl|}}_{W_h^\prime}\Bigl\lA \Rcko_h^n, \fz{1}{2} \Ecko_h^n \Bigr\rA_{W_h}.
\label{ieq:error_substitution}
\end{align}
Since the condition \eqref{cond:dt_w_Jacobian} is satisfied,
Lemma~\ref{lem:estimates_r_R} implies that
\begin{subequations}\label{ieqs:r_R}
\begin{align}
\|\rcko_h^n\|_{-1}^2
& \le c_s \|\Ecko_h^n\|_0^2 + c_s^\prime \Bigl[ \Delta t \|\ucko\|_{Z^2(t^{n-1},t^n)}^2 + h^2 \Bigl( \fz{1}{\Delta t} \|(\ucko,p)\|_{H^1(t^{n-1},t^n; H^2\times H^1)}^2 +1 \Bigr) \Bigr],
\\
{\!\!\phantom{\Bigl|}}_{W_h^\prime} \Bigl\lA \Rcko_h^n, \fz{1}{2}\Ecko_h^n\Bigr\rA_{W_h}
& \le c_s \|\Ecko_h^n\|_0^2 + \fz{\nu}{2\alpha_1^2} \|\ecko_h^n\|_1^2 + \fz{\varepsilon}{4} |\Ecko_h^n|_1^2 - \fz{1}{4}\|(\tr\Ecko_h^n) \Ecko_h^n\|_0^2 \notag\\
& \quad + c_s^\prime \Bigl[ \Delta t \|\Ccko\|_{Z^2(t^{n-1},t^n)}^2
+ h^2 \Bigl( \fz{1}{\Delta t} \|\Ccko\|_{Z^2(t^{n-1},t^n)}^2 + 1 \Bigr) \Bigr].
\end{align}
\end{subequations}
Combining~\eqref{ieqs:r_R} with~\eqref{ieq:error_substitution}, we obtain
\begin{align}
& \ol{D}_{\Delta t} \Bigl( \fz{1}{2} \|\ecko_h^n\|_0^2 + \fz{1}{4} \|\Ecko_h^n\|_0^2 \Bigr) + \fz{\nu}{2\alpha_1^2}\|\ecko_h^n\|_1^2 + \delta_0 |\epsilon_h^n|_h^2 + \fz{\varepsilon}{4} |\Ecko_h^n|_1^2 + \fz{1}{4}\|(\tr\Ecko_h^n) \Ecko_h^n\|_0^2 \notag\\
& \le c_s \Bigl( \fz{1}{2}\|\ecko_h^{n-1}\|_0^2 + \fz{1}{4}\|\Ecko_h^{n-1}\|_0^2 + \fz{1}{4}\|\Ecko_h^n\|_0^2 \Bigr) \notag\\
& \quad + c_s^\prime \Bigl[ \Delta t \|(\ucko,\Ccko)\|_{Z^2(t^{n-1},t^n)}^2 + h^2 \Bigl\{ \fz{1}{\Delta t} \bigl( \|(\ucko,p)\|_{H^1(t^{n-1},t^n; H^2\times H^1)}^2 + \|\Ccko\|_{Z^2(t^{n-1},t^n)}^2 \bigr) + 1 \Bigr\} \Bigr].
\label{ieq:error_substitution_1}
\end{align}
From \eqref{cond:ht} and \eqref{cond:t0} it holds that $\Delta t \in (0,1/(2 c_s)]$.
As for the initial value we have
\begin{align*}
(\ecko_h^0, \Ecko_h^0) = (\ucko_h^0, \Ccko_h^0) - (\hat{\ucko}_h^0, \check{\Ccko}_h^0) = ([\Pi_h^{\rm S} ({\bf 0}, -p^0)]_1, {\bf 0}) = ([(I-\Pi_h^{\rm S}) ({\bf 0}, p^0)]_1, {\bf 0}),
\end{align*}
which derives the estimates,
\begin{align}
\|\ecko_h^0\|_0 \le \alpha_3 h \|(0,p^0)\|_{H^2\times H^1} = \alpha_3 h \|p\|_{C(H^1)},\quad \|\Ecko_h^0\|_0=0.
\label{ieq:eh0_Eh0}
\end{align}
By applying Lemma~\ref{lem:Gronwall} to~\eqref{ieq:error_substitution_1} with
\begin{align*}
x^n & = \fz{1}{2} \|\ecko_h^n\|_0^2 + \fz{1}{4} \|\Ecko_h^n\|_0^2, \quad
y^n = \fz{\nu}{2\alpha_1^2}\|\ecko_h^n\|_1^2 + \delta_0 |\epsilon_h^n|_h^2 + \fz{\varepsilon}{4} |\Ecko_h^n|_1^2 + \fz{1}{4}\|(\tr\Ecko_h^n) \Ecko_h^n\|_0^2, \quad
a_0 = a_1 = c_s, \\
b^n & = c_s^\prime \Bigl[ \Delta t \|(\ucko,\Ccko)\|_{Z^2(t^{n-1},t^n)}^2 + h^2 \Bigl\{ \fz{1}{\Delta t} \bigl( \|(\ucko,p)\|_{H^1(t^{n-1},t^n; H^2\times H^1)}^2 + \|\Ccko\|_{Z^2(t^{n-1},t^n)}^2 \bigr) + 1 \Bigr\} \Bigr],
\end{align*}
and~\eqref{ieq:eh0_Eh0}, there exists a positive constant
\[
\tilde{c}_{\dagger} = c \exp(3c_sT/2) \bigl[ \|p\|_{C(H^1)} + \sqrt{c_s^\prime} \bigl( \|(\ucko,\Ccko)\|_{Z^2}+\|(\ucko,p)\|_{H^1(H^2\times H^1)}+ \sqrt{T} \bigr) \bigr]
\]
independent of~$\varepsilon$ such that
\begin{align}
\|\ecko_h\|_{\ell^{\infty}(L^2)},  \ \sqrt{\nu} \|\ecko_h\|_{\ell^2(H^1)}, \  |\epsilon_h|_{\ell^2(|.|_h)}, \  \n{\ell^{\infty}(L^2)}{\Ecko_h}, \  \sqrt{\varepsilon} |\Ecko_h|_{\ell^2(H^1)}, \ \bigl\| (\tr\Ecko_h) \Ecko_h \bigr\|_{\ell^2(L^2)}
 \le \tilde{c}_{\dagger} (h +\Delta t).
\label{ieq:error_estimates_h}
\end{align}
Hence, we obtain~\eqref{ieq:error_estimates} from~\eqref{ieq:error_estimates_h} and the estimates,
\begin{align*}
\|\ucko_h^n - \ucko^n\|_k & \le \|\ecko_h^n\|_k + \|\etacko^n\|_1 \le \|\ecko_h^n\|_k + \alpha_3 h \|(\ucko, p)\|_{C(H^2\times H^1)},\\
|p_h^n - p^n|_h & \le |\epsilon_h^n|_h + |\hat{p}_h^n - p^n|_h \le |\epsilon_h^n|_h + \alpha_3 h \|(\ucko, p)\|_{C(H^2\times H^1)},\\
\|\Ccko_h^n - \Ccko^n\|_k & \le \|\Ecko_h^n\|_k + \|\Xicko^n\|_k \le \|\Ecko_h^n\|_k + \alpha_{2(k+1)} h \|\Ccko\|_{C(H^{k+1})},\\
\|\tr (\Ccko_h^n-\Ccko^n) (\Ccko_h^n-\Ccko^n)\|_0 & = \|\tr (\Ecko_h^n-\Xicko^n) (\Ecko_h^n-\Xicko^n)\|_0 \\
&
\le
\| (\tr \Ecko_h^n) \Ecko_h^n \|_0
+ \| (\tr \Xicko^n) \Ecko_h^n \|_0
+ \| (\tr \Ecko_h^n) \Xicko^n \|_0
+ \| (\tr \Xicko^n) \Xicko^n \|_0 \\
& \le \| (\tr \Ecko_h^n) \Ecko_h^n \|_0 + c_s h (\|\Ecko_h^n\|_0+1),
\end{align*}
for~$k =0$ and~$1$.
\par
When $\varepsilon = 0$, \eqref{ieq:error_estimates} is still valid, since $\Rcko_{h3}^n$ vanishes and $c_\dagger$ is independent of~$\varepsilon$.
\qed
%
%
%
%
%
%
%
%
%
%
%
%
\section{Uniqueness of the solution}\label{sec:uniqueness}
In this section we present and prove the result on the uniqueness of the solution of scheme~\eqref{nonlin_scheme}.
Let us remind that the function $D(h)$ has been defined in \eqref{Dofh}.
\begin{prop}[uniqueness]\label{prop:uniqueness}
Suppose Hypotheses~\ref{hyp:w} and~\ref{hyp:regularity} hold.
Then, for any pair~$(h, \Delta t)$ satisfying the following condition~\eqref{ieqs:peps} or~\eqref{ieqs:peps0}, the solution of scheme~\eqref{nonlin_scheme} with~\eqref{cond:if} is unique.
\\
(i)~When $\varepsilon > 0$,
\begin{align}
  h\in (0,h_\star], \quad \Delta t \le D(h)^{-2},
\label{ieqs:peps}
\end{align}
where the constant~$h_\star$ is defined by~\eqref{def:h_star} below.
\\
(ii)~When $\varepsilon = 0$,
\begin{align}
  h\in (0,\bar{h}_\star], \quad \Delta t\le \bar{c}_\star h,
\label{ieqs:peps0}
\end{align}
where the constants~$\bar{h}_\star$ and $\bar{c}_\star$ are defined by~\eqref{def:h_star_bar} and~\eqref{def:c_star_bar} below.
\end{prop}
%
\par
The proof is given after preparing the next lemma.
\begin{lemma}\label{lem:boundedness}
Suppose Hypotheses~\ref{hyp:w} and~\ref{hyp:regularity} hold.
Then, for any pair~$(h, \Delta t)$ satisfying the following condition \eqref{ieqs:eps} or \eqref{ieqs:eps0}, any solution~$(\ucko_h, p_h, \Ccko_h)$ of scheme~\eqref{nonlin_scheme} with~\eqref{cond:if} satisfies
\begin{align}
  \|\Ccko_h\|_{\ell^\infty(L^\infty)}\le c_{c},
  \qquad
  \|\ucko_h\|_{\ell^\infty(L^\infty)}\le c_{u},
\label{ieqs:bddness}
\end{align}
where $c_c$ and $c_u$ are positive constants independent of $h$ and $\Delta t$ defined just below.
\\
(i)~When $\varepsilon > 0$,
\begin{align}
  h\in (0, h_\dagger], \quad \Delta t \le D(h)^{-2} ,
\label{ieqs:eps}
\end{align}
where $h_\dagger$ is defined by~\eqref{def:h_ast} below.
Furthermore, $c_c=c_{\dagger c}$ and~$c_u=c_{\dagger u}$, which are defined by \eqref{def:c_ast_c} and~\eqref{def:c_ast_u}.
\\
(ii)~When $\varepsilon = 0$,
\begin{align}
  h \in (0, \bar{h}_\dagger], \quad \Delta t \le h ,
\label{ieqs:eps0}
\end{align}
where $\bar{h}_\dagger$ is defined by~\eqref{def:h_ast_bar} below.
Furthermore,  $c_c=\bar{c}_{\dagger c}$ and~$c_u=\bar{c}_{\dagger u}$, which are defined by \eqref{def:c_ast_c_bar} and~\eqref{def:c_ast_u_bar}.
\end{lemma}
\begin{proof}
Let $n\in \{0,\ldots,N_T\}$ be fixed arbitrarily, and let $h_0$, $\Delta t_0$ and $\tilde{c}_\dagger$ be the positive constants in the statement of Theorem~\ref{thm:error_estimates} and in~\eqref{ieq:error_estimates_h}.
We fix a positive constant~$h_1 \in (0,1]$ such that
\[
h_1 \le D(h_1)^{-2} \le \Delta t_0.
\]
We prepare the following constants to be used in the proof:
\begin{subequations}
\begin{align}
\bar{h}_\dagger & \defeq \min\bigl\{ h_0, \Delta t_0 \bigr\},
\label{def:h_ast_bar}\\
\bar{c}_{\dagger c} & \defeq 2 \alpha_{23} \tilde{c}_\dagger + \|\Ccko\|_{C(L^\infty)},
\label{def:c_ast_c_bar}\\
\bar{c}_{\dagger u} & \defeq \alpha_{23} \bigl[ 2 \tilde{c}_\dagger + (\alpha_{21}+\alpha_3)\|(\ucko, p)\|_{C(H^2\times H^1)} \bigr]+ \|\ucko\|_{C(L^\infty)},
\label{def:c_ast_u_bar}\\
c_1 & \defeq \tilde{c}_\dagger \max\bigl\{ 1, (T + \varepsilon^{-1})^{1/2}, \nu^{-1/2}\bigr\}, \notag\\
h_\dagger & \defeq \min \{ \bar{h}_\dagger, h_1 \},
\label{def:h_ast} \\
c_{\dagger c} & \defeq \max\bigl\{ 2 \alpha_{24}c_1 + \|\Ccko\|_{C(L^\infty)}, \bar{c}_{\dagger c} \bigr\},
\label{def:c_ast_c}\\
c_{\dagger u} & \defeq \max\bigl\{ \alpha_{24} \bigl[ 2 c_1 + (\alpha_{22}+\alpha_3)\|(\ucko,p)\|_{C(H^2\times H^1)} \bigr] + \|\ucko\|_{C(L^\infty)}, \bar{c}_{\dagger u} \bigr\}.
\label{def:c_ast_u}
\end{align}
\end{subequations}
\par
Firstly, we prove~\eqref{ieqs:bddness} in case~(ii).
Since condition~\eqref{ieqs:eps0} implies~\eqref{cond:ht}, Theorem~\ref{thm:error_estimates} ensures~\eqref{ieq:error_estimates_h}.
Then, the boundedness of~$\|\Ccko_h^n\|_{0,\infty}$ is obtained as follows:
\begin{align*}
\|\Ccko_h^n\|_{0,\infty}
& \le \|\Ecko_h^n\|_{0,\infty} + \|\check{\Ccko}_h^n\|_{0,\infty}
\le \alpha_{23} h^{-1} \|\Ecko_h^n\|_0 + \|\Ccko\|_{C(L^\infty)} \\
& \le \alpha_{23} h^{-1} \tilde{c}_\dagger (\Delta t + h) + \|\Ccko\|_{C(L^\infty)}
\le 2 \alpha_{23} \tilde{c}_\dagger + \|\Ccko\|_{C(L^\infty)} \\
& = \bar{c}_{\dagger c}.
\end{align*}
Let $\check{\ucko}_h(t) \defeq (\Pi_h\ucko) (t)$ for $t\in [0,T]$.
The boundedness of~$\|\ucko_h^n\|_{0,\infty}$ is obtained as follows:
\begin{align*}
\| \ucko_h^n \|_{0,\infty}
&  \le \|\ecko_h^n\|_{0,\infty} + \|\hat{\ucko}_h^n-\check{\ucko}_h^n\|_{0,\infty} + \|\check{\ucko}_h^n\|_{0,\infty}
\le \alpha_{23} h^{-1} \bigl[ \|\ecko_h^n\|_0 + \|\hat{\ucko}_h^n-\check{\ucko}_h^n\|_0 \bigr] + \|\ucko\|_{C(L^\infty)}\\
& \le \alpha_{23} h^{-1} \bigl[ \|\ecko_h^n\|_0 + \|\hat{\ucko}_h^n-\ucko^n\|_0 + \|\ucko^n-\check{\ucko}_h^n\|_0 \bigr] + \|\ucko\|_{C(L^\infty)}\\
& \le \alpha_{23} h^{-1} \bigl[ \tilde{c}_\dagger (\Delta t + h) + \alpha_3 h \|(\ucko, p)\|_{C(H^2\times H^1)} + \alpha_{21} h \|\ucko\|_{C(H^1)} \bigr] + \|\ucko\|_{C(L^\infty)}\\
& \le \alpha_{23} \bigl[ 2 \tilde{c}_\dagger + (\alpha_{21} + \alpha_3) \|(\ucko, p)\|_{C(H^2\times H^1)} \bigr] + \|\ucko\|_{C(L^\infty)} \\
& = \bar{c}_{\dagger u}.
\end{align*}
\par
Secondly, we prove~\eqref{ieqs:bddness} in case~(i).
Since condition~\eqref{ieqs:eps} implies~\eqref{cond:ht}, the estimates~\eqref{ieq:error_estimates_h} and the definition of~$c_1$ lead to
\begin{align*}
\|\ecko_h\|_{\ell^\infty(L^2)},\ \|\ecko_h\|_{\ell^2(H^1)},\
\|\Ecko_h\|_{\ell^\infty(L^2)},\ \|\Ecko_h\|_{\ell^2(H^1)}
\le c_1 (\Delta t + h).
\end{align*}
When $\Delta t \le h$, we have~$\|\Ccko_h^n\|_{0,\infty} \le \bar{c}_{\dagger c} \le c_{\dagger c}$ and~$\|\ucko_h^n\|_{0,\infty} \le \bar{c}_{\dagger u} \le c_{\dagger u}$ from the proof in case~(ii) above.
When $(D(h)^2h^2 \le)~h \le \Delta t \le D(h)^{-2}$, we have
\begin{align*}
\|\Ccko_h^n\|_{0,\infty}
& \le \|\Ecko_h^n\|_{0,\infty} + \|\Ccko\|_{C(L^\infty)}
\le \alpha_{24} D(h) \|\Ecko_h^n\|_1 + \|\Ccko\|_{C(L^\infty)}
\le \alpha_{24} D(h) \Delta t^{-1/2} \|\Ecko_h\|_{\ell^2(H^1)} + \|\Ccko\|_{C(L^\infty)}\\
& \le \alpha_{24} c_1 D(h) ( \Delta t^{1/2} + \Delta t^{-1/2} h ) + \|\Ccko\|_{C(L^\infty)}
\le 2 \alpha_{24} c_1 + \|\Ccko\|_{C(L^\infty)} \\
& \le c_{\dagger c},\\
\|\ucko_h^n\|_{0,\infty}
& \le \|\ecko_h^n\|_{0,\infty} + \|\hat{\ucko}_h^n-\check{\ucko}_h^n\|_{0,\infty} + \|\check{\ucko}_h^n\|_{0,\infty}
\le \alpha_{24} D(h) \bigl[ \|\ecko_h^n\|_1 + \|\hat{\ucko}_h^n-\check{\ucko}_h^n\|_1 \bigr] + \|\ucko\|_{C(L^\infty)}\\
& \le \alpha_{24} D(h) \bigl[ \Delta t^{-1/2}\|\ecko_h\|_{\ell^2(H^1)} + \|\hat{\ucko}_h^n-\ucko^n\|_1 + \|\ucko^n-\check{\ucko}_h^n\|_1 \bigr] + \|\ucko\|_{C(L^\infty)} \\
& \le \alpha_{24} D(h) \bigl[ c_1 (\Delta t^{1/2} + \Delta t^{-1/2}h) + (\alpha_{22} + \alpha_3) h \|(\ucko, p)\|_{C(H^2\times H^1)} \bigr] + \|\ucko\|_{C(L^\infty)}\\
& \le \alpha_{24} \bigl[ 2 c_1 + (\alpha_{22} + \alpha_3) \|(\ucko, p)\|_{C(H^2\times H^1)} \bigr] + \|\ucko\|_{C(L^\infty)} \\
& \le c_{\dagger u}.
\end{align*}
Thus, we obtain~\eqref{ieqs:bddness}.
\end{proof}
\noindent
{\it Proof of Proposition~\ref{prop:uniqueness}.}
\ \
The definitions~\eqref{def:h_star}, \eqref{def:h_star_bar} and~\eqref{def:c_star_bar} below of the constants $h_\star$, $\bar{h}_\star$ and $c_\star$ imply  $h_\star \le h_\dagger$, $\bar{h}_\star \le \bar{h}_\dagger$ and $\bar{c}_\star \le 1$.
Hence any pair of~$(h,\Delta t)$ in Proposition~\ref{prop:uniqueness} satisfies the assumptions of Lemma~\ref{lem:boundedness} for $\varepsilon \ge 0$.
\par
Suppose $(\tilde{\ucko}_h, \tilde{p}_h, \tilde{\Ccko}_h)$ and $(\ucko_h, p_h, \Ccko_h)$ are any two solutions of scheme~\eqref{nonlin_scheme} with~\eqref{cond:if}.
Let $(\tilde{\ecko}_h, \tilde{\epsilon}_h, \tilde{\Ecko}_h) \defeq (\tilde{\ucko}_h, \tilde{p}_h, \tilde{\Ccko}_h) - (\ucko_h, p_h, \Ccko_h)$ be the difference.
Since both of $(\tilde{\ucko}_h, \tilde{p}_h, \tilde{\Ccko}_h)$ and $(\ucko_h, p_h, \Ccko_h)$ satisfy scheme~\eqref{nonlin_scheme} with \eqref{cond:if}, we have
\begin{subequations}\label{eqns:difference}
\begin{align}
\biggl( \fz{\tilde{\ecko}_h^n - \tilde{\ecko}_h^{n-1} \circ X_1^n}{\Delta t}, \vcko_h \biggr) & + \mathcal{A}_h \bigl( (\tilde{\ecko}_h^n, \tilde{\epsilon}_h^n), (\vcko_h,q_h) \bigr)
= - \bigl( (\tr \tilde{\Ecko}_h^n) \tilde{\Ecko}_h^n, \nabla \vcko_h \bigr) + {}_{V_h^\prime}\lA \tilde{\rcko}_h^n, \vcko_h\rA_{V_h},\qquad
\label{eq:difference_up} \\
\biggl( \fz{\tilde{\Ecko}_h^n - \tilde{\Ecko}_h^{n-1} \circ X_1^n}{\Delta t}, \Dcko_h \biggr) & + \varepsilon a_c (\tilde{\Ecko}_h^n, \Dcko_h) = 2 \bigl( (\nabla \tilde{\ecko}_h^n)\tilde{\Ecko}_h^n, \Dcko_h \bigr) + \bigl( (\di \tilde{\ecko}_h^n) (\tilde{\Ecko}_h^n)^{\#}, \Dcko_h \bigr) + {}_{W_h^\prime}\lA \tilde{\Rcko}_h^n, \Dcko_h\rA_{W_h},
\label{eq:difference_C}\\
& \qquad\qquad\qquad\qquad\qquad\qquad\qquad\qquad\qquad\qquad \forall (\vcko_h, q_h, \Dcko_h) \in V_h\times Q_h \times W_h, \notag
\end{align}
\end{subequations}
where
\begin{align*}
\tilde{\rcko}_h^n & \in V_h^\prime, \qquad \tilde{\Rcko}_h^n \defeq \sum_{i=1}^{5} \tilde{\Rcko}_{hi}^n \in W_h^\prime,\\
{}_{V_h^\prime}\lA \tilde{\rcko}_h^n, \vcko_h \rA_{V_h} & \defeq -\bigl( (\tr\Ccko_h^n)\tilde{\Ecko}_h^n + (\tr\tilde{\Ecko}_h^n)\Ccko_h^n, \nabla\vcko_h \bigr), \\
( \tilde{\Rcko}_{h1}^n, \Dcko_h ) & \defeq 2\bigl( (\nabla \ucko_h^n) \tilde{\Ecko}_h^n + (\nabla \tilde{\ecko}_h^n) \Ccko_h^n, \Dcko_h\bigr), \\
( \tilde{\Rcko}_{h2}^n, \Dcko_h ) & \defeq \bigl( (\di \ucko_h^n) (\tilde{\Ecko}_h^n)^{\#} + (\di \tilde{\ecko}_h^n) (\Ccko_h^n)^{\#}, \Dcko_h \bigr),\\
( \tilde{\Rcko}_{h3}^n, \Dcko_h ) & \defeq -\bigl( [\tr (\tilde{\Ecko}_h^n+\Ccko_h^n)]^2 \tilde{\Ecko}_h^n, \Dcko_h \bigr),\\
( \tilde{\Rcko}_{h4}^n, \Dcko_h ) & \defeq -\bigl( [\tr (\tilde{\Ecko}_h^n+2\Ccko_h^n)] (\tr\tilde{\Ecko}_h^n) \Ccko_h^n, \Dcko_h\bigr),\\
( \tilde{\Rcko}_{h5}^n, \Dcko_h ) & \defeq \bigl( (\tr\tilde{\Ecko}_h^n)\Icko, \Dcko_h \bigr),
\end{align*}
and $(\tilde{\ecko}_h^0, \tilde{\Ecko}_h^0)=({\bf 0}, {\bf 0})$.
Substituting~$(\tilde{\ecko}_h^n, - \tilde{\epsilon}_h^n, \fz{1}{2}\tilde{\Ecko}_h^n)$ into~$(\vcko_h, q_h, \Dcko_h)$ in~\eqref{eqns:difference} and using Lemma~\ref{lem:vanish_nonlin} and similar estimates in the derivation of~\eqref{ieq:error_substitution}, we have
\begin{align}
\ol{D}_{\Delta t} \Bigl( \fz{1}{2} \|\tilde{\ecko}_h^n\|_0^2 + \fz{1}{4} \|\tilde{\Ecko}_h^n\|_0^2 \Bigr)
& + \fz{\nu}{\alpha_1^2}\|\tilde{\ecko}_h^n\|_1^2 + \delta_0 |\tilde{\epsilon}_h^n|_h^2 + \fz{\varepsilon}{2} |\tilde{\Ecko}_h^n|_1^2
\notag\\
& \le c_w (\|\tilde{\ecko}_h^{n-1}\|_0^2 + \|\tilde{\Ecko}_h^{n-1}\|_0^2 ) + \fz{\alpha_1^2}{4\nu} \| \tilde{\rcko}_h^n \|_{-1}^2 + \Bigl( \tilde{\Rcko}_h^n, \fz{1}{2} \tilde{\Ecko}_h^n \Bigr).
\label{ieq:difference_substitution}
\end{align}
The following estimates are obtained for the functionals~$\tilde{\rcko}_h^n$ and~$\tilde{\Rcko}_h^n$:
\begin{align}
\| \tilde{\rcko}_h^n \|_{-1} \le c \| \Ccko_h^n \|_{0,\infty} \|\tilde{\Ecko}_h^n\|_0,
\qquad
\label{ieq:j}
\end{align}
\vspace*{-2\baselineskip}
\begin{subequations}\label{ieqs:J}
\begin{align}
\Bigl( \tilde{\Rcko}_{h1}^n, \fz{1}{2} \tilde{\Ecko}_h^n \Bigr), \  \Bigl( \tilde{\Rcko}_{h2}^n, \fz{1}{2} \tilde{\Ecko}_h^n \Bigr) & \le c \|\tilde{\Ecko}_h^n\|_0 \bigl( \| \ucko_h^n \|_{0,\infty} |\tilde{\Ecko}_h^n|_1 + \| \Ccko_h^n \|_{0,\infty} | \tilde{\ecko}_h^n|_1 \bigr),
\label{ieq:J1J2}
\\
\Bigl( \tilde{\Rcko}_{h3}^n, \fz{1}{2}\tilde{\Ecko}_h^n \Bigr) & \le -\fz{3}{8}\|(\tr\tilde{\Ecko}_h^n) \tilde{\Ecko}_h^n\|_0^2 + c \|\Ccko_h^n\|_{0,\infty}^2 \|\tilde{\Ecko}_h^n\|_0^2,
\label{ieq:J3}
\\
\Bigl( \tilde{\Rcko}_{h4}^n, \fz{1}{2}\tilde{\Ecko}_h^n \Bigr) & \le \fz{1}{8} \| (\tr\tilde{\Ecko}_h^n) \tilde{\Ecko}_h^n\|_0^2 + c \|\Ccko_h^n\|_{0,\infty}^2 \|\tilde{\Ecko}_h^n\|_0^2,
\label{ieq:J4}
\\
\|\tilde{\Rcko}_{h5}^n\|_0 & \le c \|\tilde{\Ecko}_h^n\|_0.
\label{ieq:J5}
\end{align}
\end{subequations}
We note that the estimates~\eqref{ieq:J1J2} are proved by the integration by parts,
\begin{align*}
\Bigl( \tilde{\Rcko}_{h1}^n, \fz{1}{2} \tilde{\Ecko}_h^n \Bigr)
& = \bigl( (\nabla \ucko_h^n)\tilde{\Ecko}_h^n, \tilde{\Ecko}_h^n \bigr) + \bigl( (\nabla \tilde{\ecko}_h^n) \Ccko_h^n, \tilde{\Ecko}_h^n\bigr)
= - \bigl( \ucko_h^n, \nabla ( \tilde{\Ecko}_h^n \tilde{\Ecko}_h^n ) \bigr) + \bigl( (\nabla \tilde{\ecko}_h^n) \Ccko_h^n, \tilde{\Ecko}_h^n\bigr)\\
& \le c \bigl( \| \ucko_h^n \|_{0,\infty} \|\tilde{\Ecko}_h^n\|_0 |\tilde{\Ecko}_h^n|_1 + \| \Ccko_h^n \|_{0,\infty} |\tilde{\ecko}_h^n|_1 \|\tilde{\Ecko}_h^n\|_0 \bigr),\\
\Bigl( \tilde{\Rcko}_{h2}^n, \fz{1}{2} \tilde{\Ecko}_h^n \Bigr)
& = \fz{1}{2} \bigl( (\di \ucko_h^n) (\tilde{\Ecko}_h^n)^{\#}, \tilde{\Ecko}_h^n \bigr) + \fz{1}{2} \bigl( (\di \tilde{\ecko}_h^n) (\Ccko_h^n)^{\#}, \tilde{\Ecko}_h^n \bigr) \\
& = -\fz{1}{2} \bigl( \ucko_h^n \nabla (\tilde{\Ecko}_h^n)^{\#}, \tilde{\Ecko}_h^n \bigr) -\fz{1}{2} \bigl( (\tilde{\Ecko}_h^n)^{\#}, \ucko_h^n \nabla \tilde{\Ecko}_h^n \bigr) + \fz{1}{2} \bigl( (\di \tilde{\ecko}_h^n) (\Ccko_h^n)^{\#}, \tilde{\Ecko}_h^n \bigr) \\
& \le c \bigl( \|\ucko_h^n\|_{0,\infty} |\tilde{\Ecko}_h^n|_1 \|\tilde{\Ecko}_h^n\|_0 +  \|\Ccko_h^n\|_{0,\infty} |\tilde{\ecko}_h^n|_1 \|\tilde{\Ecko}_h^n\|_0 \bigr),
\end{align*}
and that the other estimates~\eqref{ieq:j}, \eqref{ieq:J3}, \eqref{ieq:J4} and~\eqref{ieq:J5} are obtained similarly to~\eqref{ieq:r3}, \eqref{ieq:R8}, \eqref{ieq:R9} and~\eqref{ieq:R11}, respectively.
Applying Lemma~\ref{lem:boundedness} to~\eqref{ieq:j}, we have
\begin{align}
\| \tilde{\rcko}_h^n \|_{-1}
\le c c_c \|\tilde{\Ecko}_h^n\|_0.
\label{ieq:difference_substitution_j}
\end{align}
\par
We consider case~(i).
The estimates~\eqref{ieqs:J} and Lemma~\ref{lem:boundedness} lead to
\begin{align}
\Bigl( \tilde{\Rcko}_h^n, \fz{1}{2} \tilde{\Ecko}_h^n \Bigr)
& \le \fz{c}{\varepsilon} (c_c^2+c_u^2+1) \|\tilde{\Ecko}_h^n\|_0^2 + \fz{\nu}{2\alpha_1^2}\|\tilde{\ecko}_h^n\|_1^2 + \fz{\varepsilon}{4} |\tilde{\Ecko}_h^n|_1^2 - \fz{1}{4}\|(\tr\tilde{\Ecko}_h^n) \tilde{\Ecko}_h^n\|_0^2.
\label{ieq:difference_substitution_J}
\end{align}
Combining~\eqref{ieq:difference_substitution_j} and~\eqref{ieq:difference_substitution_J} with~\eqref{ieq:difference_substitution}, we have
\begin{align}
\ol{D}_{\Delta t} \Bigl( \fz{1}{2} \|\tilde{\ecko}_h^n\|_0^2 + \fz{1}{4} \|\tilde{\Ecko}_h^n\|_0^2 \Bigr) + \fz{\nu}{2\alpha_1^2}\|\tilde{\ecko}_h^n\|_1^2 + \delta_0 |\tilde{\epsilon}_h^n|_h^2 + \fz{\varepsilon}{4} |\tilde{\Ecko}_h^n|_1^2 + \fz{1}{4}\|(\tr\tilde{\Ecko}_h^n) \tilde{\Ecko}_h^n\|_0^2 \notag\\
\le \fz{c}{\varepsilon} (c_c^2+c_u^2+1) \Bigl( \fz{1}{4} \|\tilde{\Ecko}_h^n\|_0^2 \Bigr) + c_w \Bigl( \fz{1}{2} \|\tilde{\ecko}_h^{n-1}\|_0^2 + \fz{1}{4} \|\tilde{\Ecko}_h^{n-1}\|_0^2 \Bigr).
\label{ieq:difference_gronwall}
\end{align}
Let~$\Delta t_\star \defeq \varepsilon / [ 2c (c_c^2+c_u^2+1) ]$, and we fix a positive constant~$h_2 \in (0,1]$ such that~$D(h_2)^{-2} \le \Delta t_\star$.
We define~$h_\star$ by
\begin{align}
h_\star \defeq \min \{ h_\dagger, h_2 \}.
\label{def:h_star}
\end{align}
Condition \eqref{ieqs:peps} implies $\Delta t \le D(h_2)^{-2} \le \varepsilon/ [ 2c (c_c^2+c_u^2+1) ]~(=\Delta t_\star)$.
Applying Lemma~\ref{lem:Gronwall} to~\eqref{ieq:difference_gronwall}
with
\begin{align*}
x^n & = \fz{1}{2} \|\tilde{\ecko}_h^n\|_0^2 + \fz{1}{4} \|\tilde{\Ecko}_h^n\|_0^2, &
y^n & = \fz{\nu}{2\alpha_1^2}\|\tilde{\ecko}_h^n\|_1^2 + \delta_0 |\tilde{\epsilon}_h^n|_h^2 + \fz{\varepsilon}{4} |\tilde{\Ecko}_h^n|_1^2 + \fz{1}{4}\|(\tr\tilde{\Ecko}_h^n) \tilde{\Ecko}_h^n\|_0^2, \\
a_0 & = \fz{c}{\varepsilon} (c_c^2+c_u^2+1), \quad a_1 = 0, &
b^n & = c_w \Bigl( \fz{1}{2} \|\tilde{\ecko}_h^{n-1}\|_0^2 + \fz{1}{4} \|\tilde{\Ecko}_h^{n-1}\|_0^2 \Bigr),
\end{align*}
and using the fact~$(\tilde{\ecko}_h^0, \tilde{\Ecko}_h^0) = ( {\bf 0},{\bf 0})$, we get~$(\tilde{\ecko}_h, \tilde{\epsilon}_h, \tilde{\Ecko}_h)=({\bf 0}, 0, {\bf 0})$.
\par
We prove (ii).
In place of~\eqref{ieq:J1J2} we use the estimates,
\begin{align}
\Bigl( \tilde{\Rcko}_{h1}^n, \fz{1}{2} \tilde{\Ecko}_h^n \Bigr), \ \Bigl( \tilde{\Rcko}_{h2}^n, \fz{1}{2} \tilde{\Ecko}_h^n \Bigr)
& \le c \|\tilde{\Ecko}_h^n\|_0 \bigl( \alpha_{26} h^{-1} \| \ucko_h^n \|_{0,\infty} \|\tilde{\Ecko}_h^n\|_0 + \| \Ccko_h^n \|_{0,\infty} |\tilde{\ecko}_h^n|_1 \bigr).
\tag{\ref{ieq:J1J2}$^\prime\,$}
\label{ieq:J1J2_eps_0}
\end{align}
We define~$\bar{h}_\star$ by
\begin{align}
\bar{h}_\star \defeq \min \bigl\{ \bar{h}_\dagger, 1/c_u, c_u/c_c^2 \bigr\}.
\label{def:h_star_bar}
\end{align}
For any $h \in (0, \bar{h}_\star]$ the estimates~\eqref{ieqs:J}, Lemma~\ref{lem:boundedness} and~\eqref{def:h_star_bar} lead to
\begin{align}
\Bigl( \tilde{\Rcko}_h^n, \fz{1}{2} \tilde{\Ecko}_h^n \Bigr)
& \le c \Bigl( \fz{c_u}{h} + c_c^2 + 1 \Bigr) \|\tilde{\Ecko}_h^n\|_0^2 + \fz{\nu}{2\alpha_1^2}\|\tilde{\ecko}_h^n\|_1^2 -\fz{1}{4}\|(\tr\tilde{\Ecko}_h^n) \tilde{\Ecko}_h^n\|_0^2 \notag\\
& \le \fz{c^\prime c_u}{h} \|\tilde{\Ecko}_h^n\|_0^2 + \fz{\nu}{2\alpha_1^2}\|\tilde{\ecko}_h^n\|_1^2 -\fz{1}{4}\|(\tr\tilde{\Ecko}_h^n) \tilde{\Ecko}_h^n\|_0^2.
\label{ieq:difference_substitution_J_eps_0}
\end{align}
Combining~\eqref{ieq:difference_substitution_j} and~\eqref{ieq:difference_substitution_J_eps_0} with~\eqref{ieq:difference_substitution}, we have
\begin{align}
\ol{D}_{\Delta t} \Bigl( \fz{1}{2} \|\tilde{\ecko}_h^n\|_0^2 + \fz{1}{4} \|\tilde{\Ecko}_h^n\|_0^2 \Bigr) + \fz{\nu}{2\alpha_1^2}\|\tilde{\ecko}_h^n\|_1^2 + \delta_0 |\tilde{\epsilon}_h^n|_h^2 + \fz{1}{4}\|(\tr\tilde{\Ecko}_h^n) \tilde{\Ecko}_h^n\|_0^2
\le \fz{c c_u}{h} \Bigl( \fz{1}{4} \|\tilde{\Ecko}_h^n\|_0^2 \Bigr) + c_w \Bigl( \fz{1}{2} \|\tilde{\ecko}_h^{n-1}\|_0^2 + \fz{1}{4} \|\tilde{\Ecko}_h^{n-1}\|_0^2 \Bigr).
\label{ieq:difference_gronwall_eps_0}
\end{align}
We define~$\bar{c}_\star$ by
\begin{align}
\bar{c}_\star \defeq \min \bigl\{ 1, 1/(2 c c_u) \bigr\}.
\label{def:c_star_bar}
\end{align}
Since condition~\eqref{ieqs:peps0} implies $\Delta t \le h/(2c c_u)$, applying Lemma~\ref{lem:Gronwall} to~\eqref{ieq:difference_gronwall_eps_0}
with
\begin{align*}
x^n & = \fz{1}{2} \|\tilde{\ecko}_h^n\|_0^2 + \fz{1}{4} \|\tilde{\Ecko}_h^n\|_0^2, &
y^n & = \fz{\nu}{2\alpha_1^2}\|\tilde{\ecko}_h^n\|_1^2 + \delta_0 |\tilde{\epsilon}_h^n|_h^2 + \fz{1}{4}\|(\tr\tilde{\Ecko}_h^n) \tilde{\Ecko}_h^n\|_0^2, \\
a_0 & = \fz{c c_u}{h}, \qquad a_1 = 0, &
b^n & = c_w \Bigl( \fz{1}{2} \|\tilde{\ecko}_h^{n-1}\|_0^2 + \fz{1}{4} \|\tilde{\Ecko}_h^{n-1}\|_0^2 \Bigr),
\end{align*}
and using the fact~$(\tilde{\ecko}_h^0, \tilde{\Ecko}_h^0) = ({\bf 0}, {\bf 0})$, we obtain~$(\tilde{\ecko}_h, \tilde{\epsilon}_h, \tilde{\Ecko}_h)=({\bf 0}, 0, {\bf 0})$, which completes the proof of~(ii).
\qed
%
%
%
%
%
\section{Numerical experiments}\label{sec:numerics}
In this section we present numerical results by scheme~\eqref{nonlin_scheme} in order to confirm the theoretical convergence order.
For the detailed description of the algorithm we refer to~\cite{Miz-2015}.
\begin{example}
In problem~\eqref{model} we set $\Omega=(0, 1)^2$ and $T=0.5$, and we consider three cases for the pair of~$\nu$ and~$\varepsilon$,
\[
(\nu,\varepsilon)=(10^{-1}, 10^{-1}),\ (10^{-1}, 10^{-3}),\ (1, 0).
\]
The functions $\fko$, $\Fko$, $\ucko^0$ and $\Ccko^0$ are given such that the exact solution to~\eqref{model} is as follows:
\begin{equation}\label{exact_solution}
\begin{aligned}
\mathbf{u} (x,t) &= \( \pd{\psi}{x_2} (x,t), -\pd{\psi}{x_1} (x,t) \),\quad p(x,t) = \sin \{ \pi (x_1 + 2 x_2 + t) \},\\
C_{11}(x,t) &=\frac{1}{2} \sin^2 (\pi x_1) \sin^2 (\pi x_2) \sin \{\pi(x_1+t)\} + 1,\\
C_{22}(x,t) &=\frac{1}{2} \sin^2 (\pi x_1) \sin^2 (\pi x_2) \sin \{\pi(x_2+t)\} + 1,\\
C_{12}(x,t) &=\frac{1}{2} \sin^2 (\pi x_1) \sin^2 (\pi x_2) \sin \{\pi(x_1+x_2+t)\} \ (=C_{21}(x,t)), \\
\psi(x,t) & \defeq \frac{\sqrt{3}}{2\pi} \sin^2 (\pi x_1) \sin^2 (\pi x_2) \sin \{ \pi (x_1+x_2 + t) \}.
\end{aligned}
\end{equation}
\end{example}
\noindent Note that we set $\mathbf{w} \equiv \mathbf{u}$ in the material derivative $\textnormal{D}/\textnormal{D}t$.
\par
Since Theorem~\ref{thm:error_estimates} holds for any fixed positive constant~$\delta_0$, we simply fix $\delta_0=1$.
Let $N$ be the division number of each side of the square domain.
We set $N=32, 64, 128$ and $256$, and (re)define $h \defeq 1/N$.
The time increment is set as $\Delta t = h/2$.
\par
Let us recall that $\Pi_h^L: C(\bar{\Omega}) \to M_h$ is the {\rm Lagrange} interpolation operator.
We use the same symbol~$\Pi_h^L$ to represent the {\rm Lagrange} operators on $C(\bar{\Omega})^2$ and $C(\bar{\Omega})^{2\times 2}$.
We apply the scheme~\eqref{nonlin_scheme} with the initial conditions~\eqref{cond:if}, where $\Pi_h^L$ is employed in place of $\Pi_h$ for the choice of the initial value~$\Ccko_h^0$ in~\eqref{cond:if}.
Let us note that when the exact conformation tensor $\Ccko(t)$ belongs to~$C(\bar{\Omega})^{2\time 2}$, the error estimates~\eqref{ieq:error_estimates} in Theorem~\ref{thm:error_estimates} hold true also for the choice of initial value with~$\Pi_h^L$.
For the solution~$(\ucko_h, p_h, \Ccko_h)$ of scheme~\eqref{nonlin_scheme} and the exact solution~$(\ucko, p, \Ccko)$ given by \eqref{exact_solution} we define the relative errors~$Er\,i$, $i=1,\ldots, 6$, by
\begin{align*}
Er\,1 & =\frac{\|\ucko_h - \Pi_h^L \ucko\|_{\ell^\infty(L^2)}}{\| \Pi_h^L \ucko\|_{\ell^\infty(L^2)}}, &
Er\,2 & =\frac{\|\ucko_h - \Pi_h^L \ucko\|_{\ell^2(H^1)}}{\| \Pi_h^L \ucko\|_{\ell^2(H^1)}}, \\
Er\,3 & =\frac{\|p_h - \Pi_h^L p\|_{\ell^2(L^2)}}{\| \Pi_h^L p\|_{\ell^2(L^2)}}, &
Er\,4 & =\frac{|p_h - \Pi_h^L p|_{\ell^2(|\cdot|_h)}}{\| \Pi_h^L p\|_{\ell^2(L^2)}}, \\
Er\,5 & =\frac{\|\Ccko_h - \Pi_h^L \Ccko\|_{\ell^\infty(L^2)}}{\| \Pi_h^L \Ccko\|_{\ell^\infty(L^2)}}, &
Er\,6 & =\frac{\|\Ccko_h - \Pi_h^L \Ccko\|_{\ell^2(H^1)}}{\| \Pi_h^L \Ccko\|_{\ell^2(H^1)}}.
\end{align*}
\par
In the following we show three pairs of table and figure.
Table~\ref{table:symbols} summarizes the symbols used in the figures.
Tables~\&~Figures~\ref{table:first_case}, \ref{table:second_case} and~\ref{table:third_case} present the results for the cases~$(\nu, \varepsilon)=(10^{-1}, 10^{-1})$, $(10^{-1}, 10^{-3})$ and~$(1, 0)$, respectively.
In the tables the values of the errors and the slopes are presented, and in the figures the graphs of the errors versus $h$ in logarithmic scale are shown.
In each figure the slope of the triangle is equal to 1, which shows the convergence order $O(h)$.
\par
We can see that all the errors except~$Er\,6$ for~$(\nu,\varepsilon) = (1, 0)$ are almost of the first order in~$h$ for all the cases.
These results support Theorem~\ref{thm:error_estimates}.
In the case of~$(\nu,\varepsilon) = (1, 0)$ there is no diffusion for~$\Ccko$ in equation~\eqref{model_Ccko} and the error estimate of the conformation tensor in $\ell^2(H^1)$-seminorm disappear from~\eqref{ieq:error_estimates}.
It is, therefore, natural that the slope of~$Er\,6$ does not attain~$1$.
Although we do not have any theoretical result for $Er\,3$ at present, scheme~\eqref{nonlin_scheme} has produced convergence results also in this norm.
\clearpage
\begin{table}[!htbp]
\centering
\caption{Symbols used in the figures.}\label{table:symbols}
\begin{tabular}{cccccccc}
\toprule
\multicolumn{2}{c}{$\ucko_h$} && \multicolumn{2}{c}{$p_h$} && \multicolumn{2}{c}{$\Ccko_h$} \\ \cmidrule{1-2} \cmidrule{4-5} \cmidrule{7-8}
{\LARGE $\circ$} & {\Large $\bullet$} && $\triangle$ & {\large $\blacktriangle$} && $\Box$ & $\blacksquare$ \\
$Er\,1$ & $Er\,2$ && $Er\,3$ & $Er\,4$ && $Er\,5$ & $Er\,6$ \\
\bottomrule
\end{tabular}
\end{table}
\begin{figure}[!hbtbp]
\centering
\begin{tabular}{rrrrr}
\toprule
 $h$  &  $Er\,1$ &  slope  & $Er\,2$ & slope  \\  \midrule
 $1/32$   & $2.07 \times 10^{-2}$ & --         & $2.91 \times 10^{-2}$ & --         \\
 $1/64$   & $8.29 \times 10^{-3}$ & $1.32$ & $1.21 \times 10^{-2}$ & $1.27$  \\
 $1/128$ & $3.72 \times 10^{-3}$ & $1.16$ & $5.85 \times 10^{-3}$ & $1.05$  \\
 $1/256$ & $1.77 \times 10^{-3}$ & $1.07$ & $2.60 \times 10^{-3}$ & $1.17$  \\
\cmidrule{1-5}
 $h$  &  $Er\,3$ &  slope  & $Er\,4$ & slope \\
\cmidrule{1-5}
 $1/32$   & $6.73 \times 10^{-2}$ & --         & $5.08 \times 10^{-2}$ & -- \\
 $1/64$   & $2.06 \times 10^{-2}$ & $1.71$ & $1.86 \times 10^{-2}$ & $1.45$ \\
 $1/128$ & $6.80 \times 10^{-3}$ & $1.60$ & $8.38 \times 10^{-3}$ & $1.15$\\
 $1/256$ & $2.59 \times 10^{-3}$ & $1.39$ & $3.68 \times 10^{-3}$ & $1.19$\\
\cmidrule{1-5}
 $h$       &  $Er\,5$ &  slope  &  $Er\,6$ &  slope     \\
\cmidrule{1-5}
 $1/32$   & $1.12 \times 10^{-2}$ & --         & $4.80 \times 10^{-1}$ & --  \\
 $1/64$   & $4.33 \times 10^{-3}$ & $1.37$ & $1.66 \times 10^{-2}$ & $1.54$ \\
 $1/128$ & $1.92 \times 10^{-3}$ & $1.18$ & $6.56 \times 10^{-3}$ & $1.34$ \\
 $1/256$ & $9.09 \times 10^{-4}$ & $1.08$ & $2.90 \times 10^{-3}$ & $1.18$ \\
\midrule
\vspace{7.7cm }
\end{tabular}\qquad
\includegraphics[height=8cm]{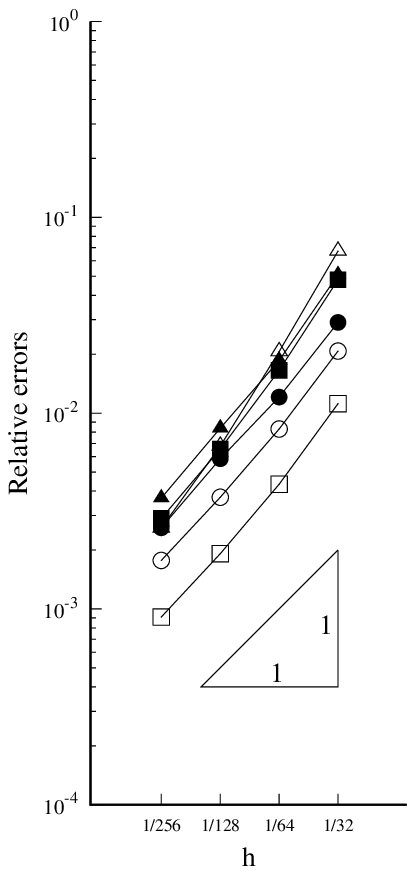}
\vspace{-7.7cm}
\captionsetup{labelformat=tableandpicture}
\caption{Errors and slopes for $(\nu,\varepsilon)=(10^{-1},10^{-1})$.}\label{table:first_case}
\end{figure}
\begin{figure}[!htbp]
\centering
\begin{tabular}{rrrrr}\toprule
 $h$  &  $Er\,1$ &  slope  & $Er\,2$ & slope  \\  \midrule
 $1/32$   & $1.75 \times 10^{-2}$ & --         & $2.71 \times 10^{-2}$ & --         \\
 $1/64$   & $6.74 \times 10^{-3}$ & $1.37$ & $1.12 \times 10^{-2}$ & $1.28$  \\
 $1/128$ & $2.91 \times 10^{-3}$ & $1.21$ & $5.49 \times 10^{-3}$ & $1.03$  \\
 $1/256$ & $1.37 \times 10^{-3}$ & $1.09$ & $2.44 \times 10^{-3}$ & $1.17$  \\
\cmidrule{1-5}
 $h$  &  $Er\,3$ &  slope  & $Er\,4$ & slope \\
\cmidrule{1-5}
 $1/32$   & $9.77 \times 10^{-2}$ & --         & $6.56 \times 10^{-2}$ & -- \\
 $1/64$   & $3.17 \times 10^{-2}$ & $1.62$ & $2.22 \times 10^{-2}$ & $1.56$ \\
 $1/128$ & $1.02 \times 10^{-2}$ & $1.63$ & $9.01 \times 10^{-3}$ & $1.30$\\
 $1/256$ & $3.62 \times 10^{-3}$ & $1.50$ & $3.78 \times 10^{-3}$ & $1.25$\\
\cmidrule{1-5}
 $h$       &  $Er\,5$ &  slope  &  $Er\,6$ &  slope     \\
\cmidrule{1-5}
 $1/32$   & $2.06 \times 10^{-2}$ & --         & $2.76 \times 10^{-1}$ & --  \\
 $1/64$   & $7.36 \times 10^{-3}$ & $1.49$ & $1.16 \times 10^{-1}$ & $1.25$ \\
 $1/128$ & $2.93 \times 10^{-3}$ & $1.33$ & $4.40 \times 10^{-2}$ & $1.40$ \\
 $1/256$ & $1.31 \times 10^{-3}$ & $1.17$ & $1.51 \times 10^{-2}$ & $1.54$ \\
\midrule
\vspace{7.7cm }
\end{tabular}\qquad
\includegraphics[height=8cm]{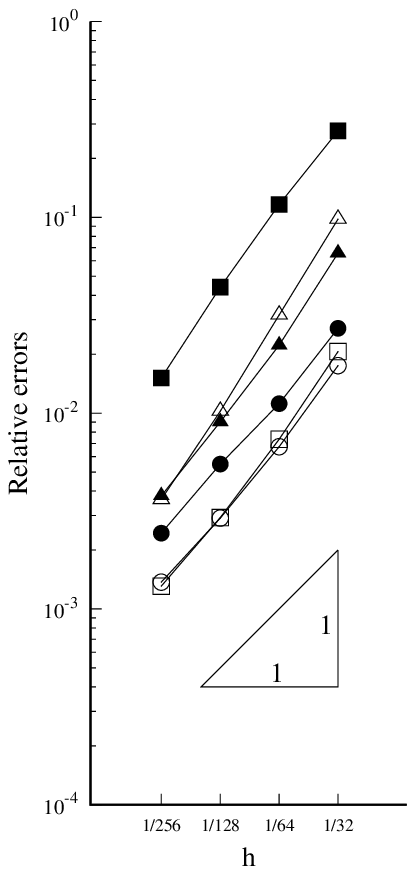}
\vspace{-7.7cm }
\captionsetup{labelformat=tableandpicture}
\caption{Errors and slopes for $(\nu,\varepsilon)=(10^{-1},10^{-3})$.}\label{table:second_case}
\end{figure}
%
%
%
%
%
%
%
%
%
%
%
%
\begin{figure}[!htbp]
\centering
\begin{tabular}{rrrrr}\toprule
  $h$  &  $Er\,1$ &  slope  & $Er\,2$ & slope  \\  \midrule
 $1/32$   & $1.36 \times 10^{-2}$ & --         & $2.30 \times 10^{-2}$ & --         \\
 $1/64$   & $4.26 \times 10^{-3}$ & $1.67$ & $9.68 \times 10^{-3}$ & $1.25$  \\
 $1/128$ & $1.40 \times 10^{-3}$ & $1.60$ & $4.84 \times 10^{-3}$ & $1.00$  \\
 $1/256$ & $5.15 \times 10^{-4}$ & $1.44$ & $2.08 \times 10^{-3}$ & $1.22$  \\
\cmidrule{1-5}
 $h$  &  $Er\,3$ &  slope  & $Er\,4$ & slope \\
\cmidrule{1-5}
 $1/32$   & $2.03 \times 10^{-1}$ & --         & $9.39 \times 10^{-2}$ & -- \\
 $1/64$   & $6.98 \times 10^{-2}$ & $1.54$ & $3.00 \times 10^{-2}$ & $1.65$ \\
 $1/128$ & $2.16 \times 10^{-2}$ & $1.69$ & $1.19 \times 10^{-2}$ & $1.34$\\
 $1/256$ & $6.86 \times 10^{-3}$ & $1.66$ & $5.05 \times 10^{-3}$ & $1.23$\\
\cmidrule{1-5}
 $h$       &  $Er\,5$ &  slope  &  $Er\,6$ &  slope     \\
\cmidrule{1-5}
 $1/32$   & $2.13 \times 10^{-2}$ & --         & $6.71 \times 10^{-1}$ & --  \\
 $1/64$   & $7.64 \times 10^{-3}$ & $1.48$ & $5.89 \times 10^{-1}$ & $0.19$ \\
 $1/128$ & $2.81 \times 10^{-3}$ & $1.44$ & $4.51 \times 10^{-1}$ & $0.38$ \\
 $1/256$ & $1.11 \times 10^{-3}$ & $1.37$ & $3.08 \times 10^{-1}$ & $0.55$ \\
\midrule
\vspace{7.7cm}
\end{tabular}\qquad
\includegraphics[height=8cm]{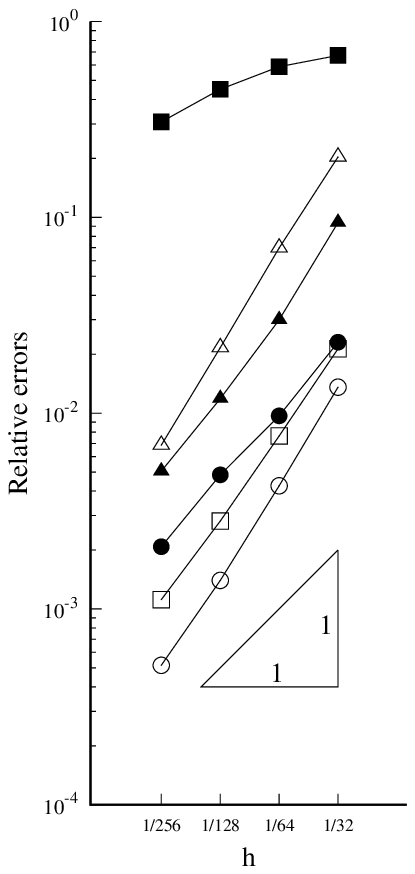}
\captionsetup{labelformat=tableandpicture}
\vspace{-7.7cm}
\caption{Errors and slopes for $(\nu,\varepsilon)=(1,0)$.}\label{table:third_case}
\end{figure}
%
%
%
%
%
%
%
%
%
%
%
%
\section{Conclusions}\label{sec:conclusions}
We have presented a nonlinear stabilized {Lagrange}--{Galerkin} scheme~\eqref{nonlin_scheme} for the Oseen-type {Peterlin} viscoelastic model.
The scheme employs the conforming linear finite elements for all unknowns, velocity, pressure and conformation tensor, together with {Brezzi}--{Pitk\"{a}ranta}'s stabilization method.
In Theorem~\ref{thm:error_estimates} we have established error estimates with the optimal convergence order, which remain true even for~$\varepsilon=0$.
We have also presented the result on the uniqueness of the solution of the scheme in Proposition~\ref{prop:uniqueness}.
It is noted that any solution of the scheme converges to the exact solution
without any relation between~$h$ and~$\Delta t$, while the condition~\eqref{ieqs:peps} or~\eqref{ieqs:peps0} is needed for the uniqueness of the solution.
Theoretical convergence order has been confirmed by  two-dimensional numerical experiments.
\par
Although we have dealt with the stabilized scheme to reduce the number of degrees of freedom, the
extension of the results to the combination of stable pairs for the velocity and the pressure, and conventional elements for the conformation tensor,  e.g., P2/P1/P2 element,  is straightforward.
Note that our analysis of the stabilized Lagrange-Galerkin method does not require to deal with the dissipation of the discrete free energy and positive definiteness of the conformation tensor $\Ccko_h$, as it was the case of the characteristic-based scheme of {Boyaval} et al.~\cite{BoyLelMan-2009} applied to the dissipative Oldroyd-B viscoelastic model.
Since the strong solution of the Peterlin model~\eqref{model} indeed satisfies these properties, cf.~\cite{Miz-2015},
they may be a useful tool in order to extend our numerical analysis to the {Peterlin} viscoelastic model with the nonlinear convective terms in future.
\par
The extension of the presented scheme to the three-dimensional case is not straightforward due to Lemma~\ref{lem:vanish_nonlin}.
Three-dimensional problems are fully treated in a forthcoming paper, Part~II, by a linear scheme,
where the convergence with the best possible order is proved for any of $\varepsilon > 0$.
%
%
%
%
%
%
%
%
%
%
%
%
%
%
%
%
%
%
%
%
%
%
%
%
\section*{\large Acknowledgements}
\small
This research was supported by the German Science Agency (DFG) under the grants IRTG 1529 ``Mathematical Fluid Dynamics'' and TRR 146 ``Multiscale Simulation Methods for Soft Matter Systems'', and by the Japan Society for the Promotion of Science~(JSPS) under the Japanese-German Graduate Externship ``Mathematical Fluid Dynamics''.
\linebreak
{H.M.} was partially supported by the German Academic Exchange Service.
{M.L.-M.} and {H.M.} wish to thank {B.~She} (Czech Academy of Science, Prague) for fruitful discussion on the topic.
H.N. and M.T. are indebted to JSPS also for Grants-in-Aid for Young Scientists~(B), No.~26800091 and for Scientific Research~(C), No.~25400212 and Scientific Research~(S), No.~24224004, respectively.
{H.N.} is supported by Japan Science and Technology Agency~(JST), PRESTO.
%
%
%
%
%
%
%
%
%
%
%
%
%
%
%
%
%
%
%
%
%
\appendix
\renewcommand{\thesection}{A}
\setcounter{lemma}{0}
\renewcommand{\thelemma}{\thesection.\arabic{Lemma}}
\setcounter{remark}{0}
\renewcommand{\theremark}{\thesection.\arabic{Remark}}
\setcounter{figure}{0}
\renewcommand{\thefigure}{\thesection.\arabic{figure}}
\setcounter{equation}{0}
\makeatletter
  \renewcommand{\theequation}{%
  \thesection.\arabic{equation}}
  \@addtoreset{equation}{section}
\makeatother
%
%
%

\begin{thebibliography}{10}

\bibitem{AboMatWeb-2002}
M.~Aboubacar, H.~Matallah, and M.F. Webster.
\newblock Highly elastic solutions for {O}ldroyd-{B} and
  {P}han-{T}hien/{T}anner fluids with a finite volume/element method: planar
  contraction flows.
\newblock {\em Journal of Non-Newtonian Fluid Mechanics}, 103:65--103, 2002.

\bibitem{BirDotJoh-1980}
R.B. Bird, P.J. Dotson, and N.L. Johnson.
\newblock Polymer-solution rheology based on a finitely extensible
  bead-spring chain model.
\newblock {\em Journal of Non-Newtonian Fluid Mechanics}, 7:213--235, 1980.

\bibitem{BonClePic-2007}
A.~Bonito, P.~Cl\'{e}ment, and M.~Picasso.
\newblock Mathematical and numerical analysis of a simplified time-dependent
  viscoelastic flow.
\newblock {\em Numerische Mathematik}, 107:213--255, 2007.

\bibitem{BonPicLas-2006}
A.~Bonito, M.~Picasso, and M.~Laso.
\newblock Numerical simulation of 3{D} viscoelastic flows with free surfaces.
\newblock {\em Journal of Computational Physics}, 215:691--716, 2006.

\bibitem{BoyLelMan-2009}
S.~Boyaval, T.~Leli\`{e}vre, and C.~Mangoubi.
\newblock Free-energy-dissipative schemes for the {O}ldroyd-{B} model.
\newblock {\em ESAIM:~M2AN}, 43:523--561, 2009.

\bibitem{BreSco-2008}
S.C. Brenner and L.R. Scott.
\newblock {\em {T}he {M}athematical {T}heory of {F}inite {E}lement {M}ethods}.
\newblock Springer, New York, 3rd edition, 2008.

\bibitem{BreDou-1988}
F.~Brezzi and J.~{Douglas~Jr.}
\newblock Stabilized mixed methods for the {S}tokes problem.
\newblock {\em Numerische Mathematik}, 53:225--235, 1988.

\bibitem{BrePit-1984}
F.~Brezzi and J.~Pitk\"{a}ranta.
\newblock On the stabilization of finite element approximations of the {S}tokes
  equations.
\newblock In W.~Hackbusch, editor, {\em {E}fficient {S}olutions of {E}lliptic
  {S}ystems}, pages 11--19, Wiesbaden, 1984. Vieweg.

\bibitem{Cia-1978}
P.G. Ciarlet.
\newblock {\em {T}he {F}inite {E}lement {M}ethod for {E}lliptic {P}roblems}.
\newblock North-Holland, Amsterdam, 1978.

\bibitem{Cle-1975}
P.~Cl\'{e}ment.
\newblock Approximation by finite element functions using local regularization.
\newblock {\em RAIRO Analyse Num\'{e}rique}, 9:77--84, 1975.

\bibitem{CroKeu-1982}
M.J. Crochet and R.~Keunings.
\newblock Finite element analysis of die swell of a highly elastic fluid.
\newblock {\em Journal of Non-Newtonian Fluid Mechanics}, 10:339--356, 1982.

\bibitem{FatKup-2004}
R.~Fattal and R.~Kupferman.
\newblock Constitutive laws for the matrix-logarithm of the conformation
  tensor.
\newblock {\em Journal of Non-Newtonian Fluid Mechanics}, 123:281--285, 2004.

\bibitem{FatKup-2005}
R.~Fattal and R.~Kupferman.
\newblock Time-dependent simulation of viscoelastic flows at high {W}eissenberg
  number using the log-conformation representation.
\newblock {\em Journal of Non-Newtonian Fluid Mechanics}, 126:23--37, 2005.

\bibitem{Keu-1986}
R.~Keunings.
\newblock On the high {W}eissenberg number problem.
\newblock {\em Journal of Non-Newtonian Fluid Mechanics}, 20:209--226, 1986.

\bibitem{LeeXu-2006}
Y.-J. Lee and J.~Xu.
\newblock New formulations, positivity preserving discretizations and stability
  analysis for non-{N}ewtonian flow models.
\newblock {\em Computer Methods in Applied Mechanics and Engineering},
  195:1180--1206, 2006.

\bibitem{LeeXuZha-2011}
Y.-J. Lee, J.~Xu, and C.-S. Zhang.
\newblock Global existence, uniqueness and optimal solvers of discretized
  viscoelastic flow models.
\newblock {\em Mathematical Models and Methods in Applied Sciences},
  21(8):1713--1732, 2011.

\bibitem{Lio-1969}
J.L. Lions.
\newblock {\em {Q}uelques {M}\'{e}thodes de {R}\'{e}solutiondes {P}robl\`{e}mes
  aux {L}imites {N}on {L}in\'{e}aires}.
\newblock Dunod et Gauthier-Villars, Paris, 1969.

\bibitem{LMNT-Peterlin_Oseen_Part_I}
M.~Luk\'{a}\v{c}ov\'{a}-Medvi\v{d}ov\'{a}, H.~Mizerov\'{a}, H.~Notsu, and
  M.~Tabata.
\newblock Numerical analysis of the {Oseen}-type {Peterlin} viscoelastic model
  by the stabilized {L}agrange--{G}alerkin method, {P}art~{II}: A linear scheme.
\newblock Submitted.

\bibitem{LukMizNec-2015}
M.~Luk\'{a}\v{c}ov\'{a}-Medvi\v{d}ov\'{a}, H.~Mizerov\'{a}, and \v{S}.
  Ne\v{c}asov\'{a}.
\newblock Global existence and uniqueness result for the diffusive {P}eterlin
  viscoelastic model.
\newblock {\em Nonlinear Analysis: Theory, Methods \& Applications},
  120:154--170, 2015.

\bibitem{LNS-2015}
M.~Luk\'{a}\v{c}ov\'{a}-Medvi\v{d}ov\'{a}, H.~Notsu, and B.~She.
\newblock Energy dissipative characteristic schemes for the diffusive
  {O}ldroyd-{B} viscoelastic fluid.
\newblock {\em International Journal for Numerical Methods in Fluids}, 2015.
\newblock Published online. DOI:~10.1002/fld.4195.

\bibitem{LMNR-2016}
M.~Luk\'{a}\v{c}ov\'{a}-Medvi\v{d}ov\'{a}, H.~Mizerov\'a, \v{S}.~Ne\v{c}asov\'a, and M.~Renardy.
\newblock Global existence result for the generalized Peterlin
viscoelastic model.
\newblock Submitted to {\em SIAM Journal of Mathematical Analysis}, 2016.

\bibitem{MarCro-1987}
J.M. Marchal and M.J. Crochet.
\newblock A new mixed finite element for calculating viscoelastic flow.
\newblock {\em Journal of Non-Newtonian Fluid Mechanics}, 26:77--114, 1987.

\bibitem{Miz-2015}
H.~Mizerov\'{a}.
\newblock Analysis and numerical solution of the {P}eterlin viscoelastic model.
\newblock 2015.
\newblock PhD thesis, University of Mainz, Germany.

\bibitem{NadSeq-2007}
L.~Nadau and A.~Sequeira.
\newblock Numerical simulations of shear-dependent viscoelastic flows with a
  combined finite element-finite volume method.
\newblock {\em Computers \& Mathematics with Applications}, 53:547--568, 2007.

\bibitem{Nec-1967}
J.~Ne\v{c}as.
\newblock {\em {L}es {M}\'ethods {D}irectes en {T}h\'eories des
  {\'E}quations {E}lliptiques}.
\newblock Masson, Paris, 1967.

\bibitem{NT-NCP}
H.~Notsu and M.~Tabata.
\newblock Error estimates of stable and stabilized {L}agrange--{G}alerkin
  schemes for natural convection problems.
\newblock {\em {\rm {a}rXiv:1511.01234~[math.NA]}}.

\bibitem{NT-2015-JSC}
H.~Notsu and M.~Tabata.
\newblock Error estimates of a pressure-stabilized characteristics finite
  element scheme for the {O}seen equations.
\newblock {\em Journal of Scientific Computing}, 65(3):940--955, 2015.

\bibitem{NT-2016-M2AN}
H.~Notsu and M.~Tabata.
\newblock Error estimates of a stabilized {L}agrange--{G}alerkin scheme for the
  {N}avier--{S}tokes equations.
\newblock {\em ESAIM:~M2AN}, 50(2):361--380, 2016.

\bibitem{Pet-1966}
A.~Peterlin.
\newblock Hydrodynamics of macromolecules in a velocity field with longitudinal
  gradient.
\newblock {\em Journal of Polymer Science Part~B: Polymer Letters}, 4:287--291,
  1966.

\bibitem{PicRap-2001}
M.~Picasso and J.~Rappaz.
\newblock Existence, a priori and a posteriori error estimates for a nonlinear
  three-field problem arising from {O}ldroyd-{B} viscoelastic flows.
\newblock {\em ESAIM:~M2AN}, 35:879--897, 2001.

\bibitem{Ren-2000}
M.~Renardy.
\newblock {\em {M}athematical {A}nalysis of {V}iscoelastic {F}lows}.
\newblock CBMS-NSF Conference Series in Applied Mathematics 73. SIAM, New York,
  2000.

\bibitem{Ren-2008}
M.~Renardy.
\newblock Mathematical analysis of viscoelastic fluids.
\newblock In {\em {H}andbook of {D}ifferential {E}quations:~{E}volutionary
  {E}quations}, volume~4, pages 229--265, Amsterdam, 2008. North-Holland.

\bibitem{Ren-2010}
M.~Renardy. 
\newblock The mathematics of myth: Yield stress behaviour as a limit of non-monotone constitutive theories. 
\newblock {\em Journal of Non-Newtonian Fluid Mechanics}, 165:519--526, 2010.

\bibitem{RenWan-2015}
M.~Renardy and T.~Wang.
\newblock Large amplitude oscillatory shear flows for a model of a thixotropic
  yield stress fluid.
\newblock {\em Journal of Non-Newtonian Fluid Mechanics}, 222:1--17, 2015.

\bibitem{RuiTab-2002}
H.~Rui and M.~Tabata.
\newblock A second order characteristic finite element scheme for
  convection-diffusion problems.
\newblock {\em Numerische Mathematik}, 92:161--177, 2002.

\bibitem{TabTag-2005}
M.~Tabata and D.~Tagami.
\newblock Error estimates of finite element methods for nonstationary thermal
  convection problems with temperature-dependent coefficients.
\newblock {\em Numerische Mathematik}, 100:351--372, 2005.

\bibitem{TabUch-2015-NS}
M.~Tabata and S.~Uchiumi.
\newblock An exactly computable {L}agrange--{G}alerkin scheme for the {N}avier--{S}tokes equations and its error estimates.
\newblock To appear in {\em Mathematics of Computation.}

\bibitem{Tem-1984}
R.~Temam.
\newblock {\em {N}avier--{S}tokes {E}quations}.
\newblock North-Holland, Amsterdam, 1984.

\bibitem{WapKeuLeg-2000}
P.~Wapperom, R.~Keunings, and V.~Legat.
\newblock The backward-tracking {L}agrangian particle method for transient
  viscoelastic flows.
\newblock {\em Journal of Non-Newtonian Fluid Mechanics}, 91:273--295, 2000.

\end{thebibliography}
 \newcommand{\noop}[1]{}

%
%
%
%
%
%
%
\end{document}